\documentclass[letterpaper,11pt,twoside,keywordsasfootnote,addressatend,noinfoline]{article}
\usepackage{fullpage}
\usepackage[english]{babel}
\usepackage{amssymb}
\usepackage{amsmath}
\usepackage{theorem}
\usepackage{epsfig}
\usepackage{imsart}
\usepackage{color}

\theorembodyfont{\normalfont}
\newtheorem{theorem}{Theorem}

\newtheorem{lemma}[theorem]{Lemma}

\newenvironment{proof}{\noindent{\scshape Proof.}}{\hspace*{2mm} $\square$}

\newcommand{\Z}{\mathbb{Z}}
\newcommand{\R}{\mathbb{R}}
\newcommand{\ind}{\mathbf{1}}
\newcommand{\ep}{\epsilon}

\DeclareMathOperator{\card}{card \,}
\DeclareMathOperator{\poisson}{Poisson \,}
\DeclareMathOperator{\binomial}{Binomial \,}
\DeclareMathOperator{\exponential}{Exponential \,}
\DeclareMathOperator{\gammadist}{Gamma \,}

\begin{document}

\begin{frontmatter}

\title     {Some rigorous results for the \\ stacked contact process}
\runtitle  {Some rigorous results for the stacked contact process}
\author    {Nicolas Lanchier and Yuan Zhang}
\runauthor {Nicolas Lanchier and Yuan Zhang}
\address   {School of Mathematical and Statistical Sciences, \\ Arizona State University, \\ Tempe, AZ 85287, USA. \\ E-mail: nlanchie@asu.edu}
\address   {Department of Mathematics, \\ Duke University, \\ Durham, North Carolina, USA. \\ E-mail: yzhang@math.duke.edu}

\begin{abstract} \ \
 The stacked contact process is a stochastic model for the spread of an infection within a population of hosts located on the~$d$-dimensional integer lattice.
 Regardless of whether they are healthy or infected, hosts give birth and die at the same rate and in accordance to the evolution rules of the neutral multitype contact process.
 The infection is transmitted both vertically from infected parents to their offspring and horizontally from infected hosts to nearby healthy hosts.
 The population survives if and only if the common birth rate of healthy and infected hosts exceeds the critical value of the basic contact process.
 The main purpose of this work is to study the existence of a phase transition between extinction and persistence of the infection in the parameter region where the hosts survive.
\end{abstract}



\end{frontmatter}


\section{Introduction}
\label{sec:intro}

\indent This paper is concerned with the stacked contact process which has been recently introduced and studied numerically in~\cite{court_blythe_allen_2012}.
 This process is a stochastic model for the spread of an infection within a population of hosts and is based on the framework of interacting particle systems.
 The model assumes that all the hosts give birth and die at the same rate regardless of whether they are healthy or infected, and that the infection is transmitted both
 vertically from infected parents to their offspring and horizontally from infected hosts to nearby healthy hosts.
 More precisely, the state of the process at time~$t$ is a spatial configuration
 $$ \xi_t : \Z^d \longrightarrow \{0, 1, 2 \} $$
 where state~0 means empty, state~1 means occupied by a healthy host, and state~2 means occupied by an infected host.
 The inclusion of an explicit spatial structure in the form of local interactions, meaning that hosts can only interact with nearby hosts, is another important component
 of the model.
 In particular, the dynamics is built under the assumption that the state at vertex~$x$ is updated at a rate that only depends on the state at~$x$ and in the neighborhood
 $$ \begin{array}{l} N_x \ := \ \{y \in \Z^d : y \neq x \ \hbox{and} \ \max_{i = 1, 2, \ldots, d} \,|y_i - x_i| \leq L \}. \end{array} $$
 Here, the parameter~$L$ is an integer which is referred to as the range of the interactions.
 From this collection of interaction neighborhoods, one defines the fraction of neighbors of every vertex~$x$ which are in state~$j$ as
 $$ f_j (x, \xi) \ := \ \card \{y \in N_x : \xi (y) = j \} / \card N_x. $$
 The transition rates at~$x$ are then given by
 $$ \begin{array}{rclcrcl}
     0 \ \to \ 1 & \hbox{at rate} & \lambda_1 \,f_1 (x, \xi) & \quad & 1 \ \to \ 0 & \hbox{at rate} & 1 \vspace*{2pt} \\
     0 \ \to \ 2 & \hbox{at rate} & \lambda_1 \,f_2 (x, \xi) & \quad & 2 \ \to \ 0 & \hbox{at rate} & 1 \vspace*{2pt} \\
     1 \ \to \ 2 & \hbox{at rate} & \lambda_2 \,f_2 (x, \xi) & \quad & 2 \ \to \ 1 & \hbox{at rate} & \delta. \end{array} $$
 The first four transition rates at the top indicate that healthy and infected hosts give birth at the same rate~$\lambda_1$ and die at the same rate one.
 An offspring produced at~$x$ is sent to a vertex chosen uniformly at random from the interaction neighborhood~$N_x$ but the birth is suppressed when the target
 site is already occupied.
 Note that the offspring is always of the same type as its parent.
 The process described exclusively by these four transitions is the multitype contact process, completely analyzed in~\cite{neuhauser_1992} under the assumption that
 both types die at the same rate.
 The stacked contact process includes two additional transitions in which individuals can also change type:
 an infected host chooses a vertex at random from its neighborhood at rate~$\lambda_2$ and, when this vertex is occupied by a healthy host, infects this host,
 which corresponds to a horizontal transmission of the infection, and an infected host recovers at the spontaneous rate~$\delta$. \\


\noindent {\bf Main results in the general case} -- Interacting particle systems are ideally suited to understand the role of space but are often difficult to study
 due to the inclusion of local interactions that create spatial correlations;
 the smaller the dimension and the range of the interactions, the stronger these correlations.
 The main question about the process is whether the host population survives or goes extinct and, in case of survival, whether the infection persists or not,
 where we say that
 $$ \begin{array}{rcl}
             \hbox{hosts survive} & \hbox{when} & \liminf_{t \to \infty} \,P \,(\xi_t (x) \neq 0) \ > \ 0 \quad \hbox{for all} \quad x \in \Z^d \vspace*{4pt} \\
    \hbox{the infection persists} & \hbox{when} & \liminf_{t \to \infty} \,P \,(\xi_t (x) = 2) \ > \ 0 \quad \hbox{for all} \quad x \in \Z^d \end{array} $$
 for the system starting from the configuration with only infected hosts.
 Our first results are mainly qualitative but hold regardless of the spatial dimension and/or the dispersal range, while our last result gives a more detailed
 picture of the phase diagram of the process under the assumption that the range of the interactions is large, which weakens spatial correlations.

\indent Whether the host population survives or goes extinct can be easily answered by observing that hosts evolve like a basic contact process.
 Indeed, defining
 $$ \eta_t^1 (x) \ := \ \ind \{\xi_t (x) \neq 0 \} \quad \hbox{for all} \quad (x, t) \in \Z^d \times \R_+ $$
 results in a spin system with transition rates
 $$ \begin{array}{rcl}
     0 \ \to \ 1 & \hbox{at rate} & \lambda_1 \,f_1 (x, \xi) + \lambda_1 \,f_2 (x, \xi) = \lambda_1 \,f_1 (x, \eta^1) \vspace*{4pt} \\
     1 \ \to \ 0 & \hbox{at rate} & 1, \end{array} $$
 which is the contact process with birth rate~$\lambda_1$ and death rate one.
 It is known for this process that there exists a critical value~$\lambda_c \in (0, \infty)$ such that above this critical value the hosts survive whereas
 at and below this critical value the population goes extinct~\cite{bezuidenhout_grimmett_1990}.
 To avoid trivialities, we assume from now on that~$\lambda_1 > \lambda_c$ and study whether the infection persists or not.

\indent Basic coupling techniques to compare processes starting from the same configuration but with different parameters does not imply that the probability
 that the infection persist is nondecreasing with respect to the birth rate.
 However, as proved in Lemma~\ref{lem:monotone} below, a coupling argument shows that, everything else being fixed, the probability that the infection
 persists is nondecreasing with respect to the infection rate.
 This implies that there is at most one phase transition between extinction and survival of the infection at some critical value
 $$ \lambda_2^* \ = \ \lambda_2^* (\lambda_1, \delta) \ := \ \inf \,\{\lambda_2 \geq 0 : \hbox{the infection persists} \}. $$
 To study the existence of such a phase transition, we observe that, as pointed out in~\cite{court_blythe_allen_2012}, when the birth rate and the infection
 rate are equal: $\lambda_1 = \lambda_2$, the set of infected hosts again evolves like a basic contact process.
 Indeed, since in this case infected hosts give birth onto adjacent empty vertices and infect adjacent healthy hosts at the same rate, letting~$\lambda$ be
 the common value of the birth and infection rates, and defining
 $$ \eta_t^2 (x) \ := \ \ind \{\xi_t (x) = 2 \} \quad \hbox{for all} \quad (x, t) \in \Z^d \times \R_+ $$
 results in a spin system with transition rates
 $$ \begin{array}{rcl}
     0 \ \to \ 1 & \hbox{at rate} & \lambda \,f_2 (x, \xi) = \lambda \,f_1 (x, \eta^2) \vspace*{4pt} \\
     1 \ \to \ 0 & \hbox{at rate} & 1 + \delta, \end{array} $$
 which is the contact process with birth rate~$\lambda$ and death rate~$1 + \delta$.
 This implies that, when birth and infection rates are equal, the infection persists if and only if~$\lambda > (1 + \delta) \,\lambda_c$.
 Using a coupling argument to compare the process in which birth and infection rates are equal with the general stacked contact process, we can improve
 this result as stated in the following theorem.
\begin{theorem} --
\label{th:coupling}
 For all recovery rate~$\delta$,
 $$ \begin{array}{rcl}
    \hbox{the infection dies out} & \hbox{when} & \max \,(\lambda_1, \lambda_2) \leq (1 + \delta) \,\lambda_c \vspace*{3pt} \\
    \hbox{the infection persists} & \hbox{when} & \min \,(\lambda_1, \lambda_2) > (1 + \delta) \,\lambda_c. \end{array} $$
\end{theorem}
 This can be translated in terms of the critical value~$\lambda_2^*$ as follows:
\begin{equation}
\label{eq:critical}
  \begin{array}{rcl}
    \lambda_2^* (\lambda_1, \delta) \ \geq \ (1 + \delta) \,\lambda_c & \hbox{when} & \lambda_1 \ \leq \ (1 + \delta) \,\lambda_c \vspace*{4pt} \\
    \lambda_2^* (\lambda_1, \delta) \ \leq \ (1 + \delta) \,\lambda_c & \hbox{when} & \lambda_1 \ > \ (1 + \delta) \,\lambda_c. \end{array}
\end{equation}
 We now look at the remaining parameter region
 $$ \min \,(\lambda_1, \lambda_2) \ \leq \ (1 + \delta) \,\lambda_c \ < \ \max \,(\lambda_1, \lambda_2). $$
 Our next theorem gives an improvement of the second part of~\eqref{eq:critical} showing that the critical infection rate is not only finite but also
 positive for all values of the birth rate and recovery rate.
 This implies the existence of exactly one phase transition when~$\lambda_1 > (1 + \delta) \,\lambda_c$.
\begin{theorem} --
\label{th:extinction}
 For all~$\lambda_1$ and~$\delta$, we have~$\lambda_2^* (\lambda_1, \delta) > 0$.
\end{theorem}
 Similarly, the next theorem gives an improvement of the first part of~\eqref{eq:critical} showing that the critical infection rate is not only positive
 but also finite.
 This, however, only holds when the birth rate exceeds some finite universal critical value~$\lambda_1^*$ that depends on the spatial dimension and the
 range of the interactions but not on the other parameters.
\begin{theorem} --
\label{th:survival}
 There exists~$\lambda_1^* < \infty$ such that
 $$ \lambda_2^* (\lambda_1, \delta) < \infty \quad \hbox{for all} \quad \lambda_1 > \lambda_1^* \quad \hbox{and} \quad \delta \geq 0. $$
\end{theorem}
\begin{figure}[t]
\centering
\scalebox{0.40}{\input{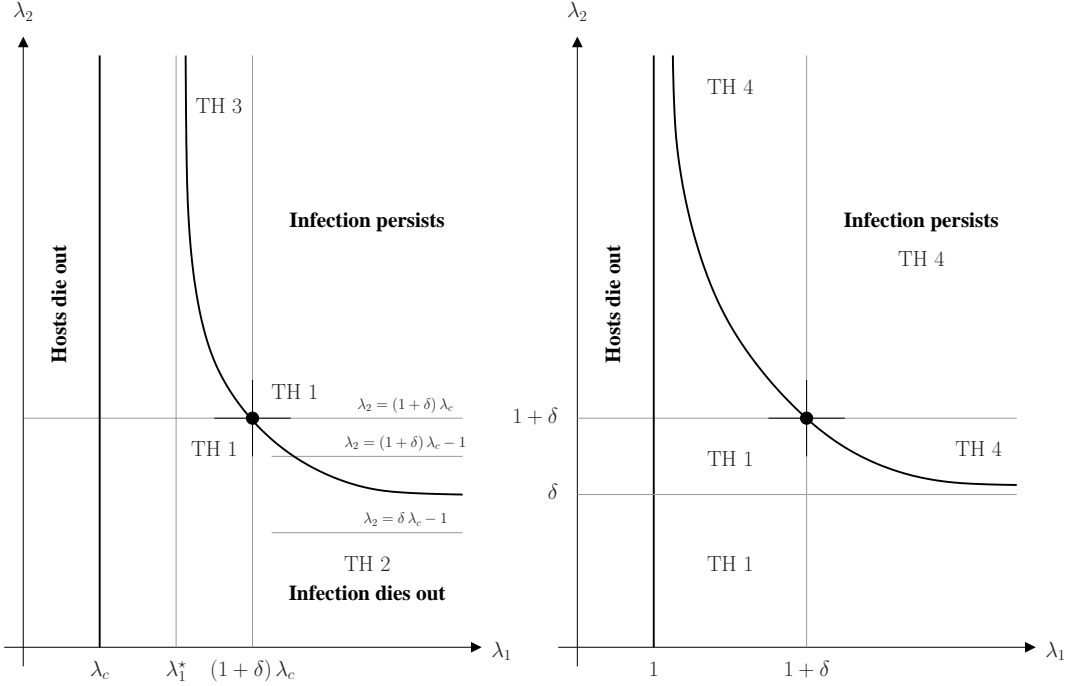}}
\caption{\upshape{Summary of our results and phase diagram of the stacked contact process with short range interactions on the left and long range of interactions on the right.}}
\label{fig:diagram}
\end{figure}
\begin{figure}[t]
\centering
\scalebox{0.36}{\input{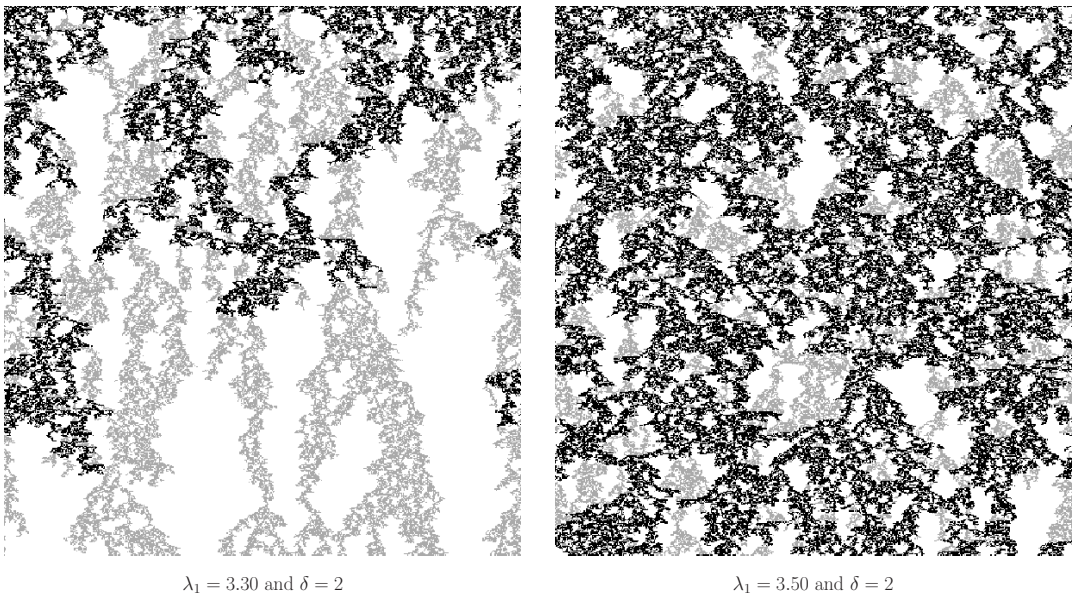}}
\caption{\upshape{Realizations of the stacked contact process on the torus~$\Z / 500 \Z$ with periodic boundary conditions until time~500 at the bottom of the pictures.
 The color code is white for empty, grey for occupied by a healthy host and black for occupied by an infected host.
 In both pictures, all vertices are initially infected and the infection rate is infinite.
 In the first realization, the birth rate is only slightly larger than the critical value of the contact process, which leads to extinction of the infection
 even when the infection rate is infinite.}}
\label{fig:simulation}
\end{figure}
 Combining the previous two theorems with~\eqref{eq:critical}, we deduce that
 $$ \lambda_2^* \in (0, \infty) \quad \hbox{for all} \quad \lambda_1 > \min ((1 + \delta) \,\lambda_c, \lambda_1^*) $$
 which also implies that, for these values of the birth rate, there is exactly one phase transition between extinction and survival of the infection.

\indent For a summary of~Theorems~\ref{th:coupling}--\ref{th:survival}, we refer to the left-hand side of Figure~\ref{fig:diagram}.
 Even though the analysis of the mean-field model below suggests that, provided the birth parameter is supercritical to ensure survival of the host
 population, the infection also persists when the infection parameter is large enough, we conjecture, as represented in the picture, that this result
 is not true for the spatial model, i.e., there exists a critical value~$\lambda_1^{\star} > \lambda_c$ such that the infection dies out when
 $$ \lambda_c \ < \ \lambda_1 \ < \ \lambda_1^{\star} \quad \hbox{regardless of} \quad \lambda_2. $$
 Here is the idea behind our intuition: when the set of occupied vertices percolates at equilibrium, it is expected that the infection persists when
 the infection rate is large.
 In contrast, when the birth rate is only slightly larger than the critical value of the contact process, the system mostly consists of very small clusters
 of hosts far from each other.
 With high probability, an infected cluster either dies out due to stochasticity or becomes healthy before it can cross another cluster of hosts, thus
 leading to the extinction of the infection.
 Figure~\ref{fig:simulation} shows an example of realization where the birth rate is only slightly supercritical and the infection dies out even when the
 infection rate is infinite, i.e., all the hosts contained in the same connected component as an infected host get instantaneously infected.
 Now, to explain the right end of the transition curve on the left-hand side of the figure, we observe that, in the limit as~$\lambda_1$ tends to infinity, all vertices
 are occupied by a host.
 More precisely, each time a host dies, it is instantaneously replaced by the offspring of a neighbor, thus leading to the transition rates
 $$ \begin{array}{rcl}
     1 \ \to \ 2 & \hbox{at rate} & (\lambda_2 + 1) \,f_2 (x, \xi) \vspace*{4pt} \\
     2 \ \to \ 1 & \hbox{at rate} &  f_1 (x, \xi) + \delta. \end{array} $$
 Using that~$\delta \leq f_1 (x, \xi) + \delta \leq 1 + \delta$, we can compare the process with two contact processes and deduce that, in the limit
 as~$\lambda_1$ goes to infinity, we have
 $$ \delta \,\lambda_c - 1 \ \leq \ \lambda_2^* (\infty, \delta) \ \leq \ (1 + \delta) \,\lambda_c - 1 $$
 as represented in the picture.
 Note however that this observation does not provide any universal lower bound for the critical value of the infection rate because the set of infected hosts
 is not monotone (for the inclusion) with respect to the birth rate. \\


\begin{figure}[t]
\centering
\scalebox{0.20}{\input{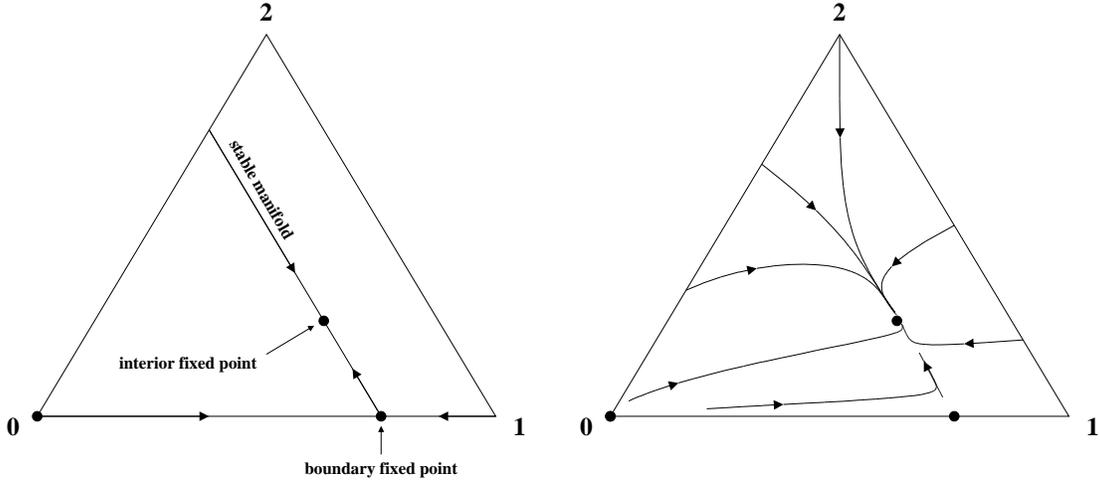}}
\caption{\upshape{Solution curves of the mean-field model when~$\lambda_1 = 4$, $\lambda_2 = 2$ and $\delta = 2$.}}
\label{fig:mean-field}
\end{figure}

\noindent {\bf Long range limits} -- Finally, we look at the spatial process when the range of the interactions is very large.
 To give the intuition behind our last theorem, we first study the non-spatial deterministic counterpart of the stochastic process called mean-field approximation,
 which is obtained by assuming that the population is well-mixing, which results in a deterministic system of ordinary differential equations
 for the density~$u_j$ of vertices in state~$j = 0, 1, 2$.
 In the case of the stacked contact process, this system can be written as
\begin{equation}
\label{eq:mean-field}
  \begin{array}{rcl}
    u_1' & = & \lambda_1 \,u_1 \,(1 - u_1 - u_2) - u_1 - \lambda_2 \,u_2 \,u_1 + \delta \,u_2 \vspace*{4pt} \\
    u_2' & = & \lambda_1 \,u_2 \,(1 - u_1 - u_2) - u_2 + \lambda_2 \,u_2 \,u_1 - \delta \,u_2. \end{array}
\end{equation}
 This system is only two-dimensional because the three densities sum up to one.
 To understand whether the hosts survive or die out, we let~$u = u_1 + u_2$ and note that
 $$ \begin{array}{rcl}
    u' \ = \ u_1' + u_2' & = & \lambda_1 \,(u_1 + u_2)(1 - u_1 - u_2) - u_1 - u_2 \vspace*{4pt} \\
                         & = & \lambda_1 \,u \,(1 - u) - u \ = \ (\lambda_1 \,(1 - u) - 1) \,u. \end{array} $$
 It is straightforward to deduce that, when~$u (0) > 0$
\begin{enumerate}
 \item if~$\lambda_1 \leq 1$, there is extinction: $\lim_{t \to \infty} \,u (t) = 0$, \vspace*{4pt}
 \item if~$\lambda_1 > 1$, there is survival: $\lim_{t \to \infty} \,u (t) = u_* := 1 - 1 / \lambda_1$.
\end{enumerate}
 To study whether the infection survives or dies out, we now assume that~$\lambda_1 > 1$ and look at the second equation in~\eqref{eq:mean-field} along the
 stable manifold~$u_1 + u_2 = u_*$. This gives
 $$ \begin{array}{rcl}
     u_2' & = & \lambda_1 \,u_2 \,(1 - u_*) - u_2 + \lambda_2 \,u_2 \,(u_* - u_2) - \delta \,u_2 \vspace*{4pt} \\
          & = & (\lambda_1 \,(1 - u_*) - 1 + \lambda_2 \,(u_* - u_2) - \delta) \,u_2 \ = \ (\lambda_2 \,(u_* - u_2) - \delta) \,u_2. \end{array} $$
 Again, it is straightforward to deduce that, when~$u_2 (0) > 0$
\begin{enumerate}
 \item if~$u_* \leq \delta / \lambda_2$, we have $\lim_{t \to \infty} \,u_2 (t) = 0$, \vspace*{4pt}
 \item if~$u_* > \delta / \lambda_2$, we have $\lim_{t \to \infty} \,u_2 (t) = u_* - \delta / \lambda_2 > 0$.
\end{enumerate}
 In conclusion, starting with~$u_2 (0) > 0$, the infection survives if and only if
\begin{equation}
\label{eq:coexist}
  \lambda_2 \,u_* \ = \ \lambda_2 \,(1 - 1 / \lambda_1) \ > \ \delta.
\end{equation}
 We refer the reader to Figure~\ref{fig:mean-field} for a schematic illustration of the three fixed points and stable manifold when~\eqref{eq:coexist} is satisfied,
 and the corresponding solution curves.
 Our last result shows that, for any set of parameters such that~\eqref{eq:coexist} holds, the infection survives as well for the stochastic process provided the
 dispersal range is large enough.
\begin{theorem} --
\label{th:long-range}
 Assume~\eqref{eq:coexist}. Then, the infection persists for all~$L$ sufficiently large.
\end{theorem}
 For an illustration, see the right-hand side of Figure~\ref{fig:diagram} where the parameter region~\eqref{eq:coexist} in which the infection persists corresponds
 to the region above the solid curve.


\section{Monotonicity and attractiveness}
\label{sec:monotonicity}

\indent This section gives some preliminary results that will be useful later.
 To prove these results as well as the theorems stated in the introduction, it is convenient to think of the stacked contact process as being generated by a collection of
 independent Poisson processes, also called Harris' graphical representation~\cite{harris_1972}.
 More precisely, the process starting from any initial configuration can be constructed using the rules and Poisson processes described in Table~\ref{tab:harris}.
 To begin with, we use this graphical representation to prove that the stacked contact process is attractive.
 Here, by attractiveness, we mean that replacing initially some healthy hosts by infected hosts can only increase the set of infected hosts at later times.
 More precisely, having two configurations of the stacked contact process~$\xi^1$ and~$\xi^2$, we write~$\xi^1 \preceq \xi^2$ whenever
\begin{align*}
  \{x \in \Z^d : \xi^1 (x) \neq 0 \} & \ \subset \ \{x \in \Z^d : \xi^2 (x) \neq 0 \} \vspace*{4pt} \\
  \{x \in \Z^d : \xi^1 (x) = 2 \} & \ \subset \ \{x \in \Z^d : \xi^2 (x) = 2 \}.
\end{align*}
 Then, the process is attractive in the sense stated in the next lemma.
\begin{lemma} --
\label{lem:attractive}
 Let~$\xi_t^1$ and~$\xi_t^2$ be two copies of the stacked contact process constructed from the same collections of independent Poisson processes. Then,
 $$ \xi^1_t \ \preceq \ \xi^2_t \ \ \hbox{for all} \ \ t \geq 0 \quad \hbox{whenever} \quad \xi^1_0 \ \preceq \ \xi^2_0. $$
\end{lemma}
\begin{proof}
 To begin with, we recall that the process that keeps track of the occupied vertices is a basic contact process.
 Since in addition both processes are constructed from the same graphical representation and start from the same set of occupied vertices,
\begin{equation}
\label{eq:attractive-1}
 \{x \in \Z^d : \xi_t^1 (x) \neq 0 \} \ = \ \{x \in \Z^d : \xi_t^2 (x) \neq 0 \} \quad \hbox{for all} \quad t \geq 0.
\end{equation}
 In particular, we only need to prove that
\begin{equation}
\label{eq:attractive-2}
  \{x \in \Z^d : \xi_t^1 (x) = 2 \} \ \subset \ \{x \in \Z^d : \xi_t^2 (x) = 2 \} \quad \hbox{for all} \quad t \geq 0.
\end{equation}
 To prove~\eqref{eq:attractive-2}, we first define the influence graph of a space-time point~$(x, t)$.
 We will show that each of the updates that occurs along the influence graph preserves the desired property.
 To define this graph, we say that there is a path from~$(y, s)$ to~$(x, t)$ if there exist
 $$ x_1 = x, x_2, \ldots, x_n = y \in \Z^d \qquad \hbox{and} \qquad s = t_0 < t_1 < \cdots < t_n = t $$
 such that the following conditions hold:
\begin{itemize}
 \item For each~$j = 1, 2, \ldots, n$, there is no death mark along $\{x_j \} \times (t_{j - 1}, t_j)$. \vspace*{4pt}
 \item For each~$j = 1, 2, \ldots, n - 1$, either~$x_j \longrightarrow x_{j + 1}$ or~$x_j \dasharrow x_{j + 1}$ at time~$t_j$.
\end{itemize}
 Then, we define the {\bf influence graph} of space-time point~$(x, t)$ as
 $$ G (x, t) \ = \ \{(y, s) \in \Z^d \times \R_+ : \hbox{there is a path from $(y, s)$ to $(x, t)$} \}. $$
 Note that the state at~$(x, t)$ can be determined from the initial configuration and the structure of the influence graph.
 Also, to prove~\eqref{eq:attractive-2}, we will prove that
\begin{equation}
\label{eq:attractive-3}
  (\xi_s^1 (y) \geq \xi_s^2 (y) \ \hbox{for all~$y \in \Z^d$ such that} \ (y, s) \in G (x, t)) \quad \hbox{for all} \quad 0 \leq s \leq t. 
\end{equation}
 Standard arguments~\cite{durrett_1995} imply that, with probability one, there can only be a finite number of paths leading to a given space-time point
 therefore the influence graph is almost surely finite which, in turn, implies that the number of death and recovery marks in the influence graph and the number of birth
 arrows and infection arrows that connect space-time points in the influence graph are almost surely finite.
 In particular, they can be ordered chronologically so the result can be proved by checking that each of the successive events occurring along the influence graph going
 forward in time preserves the relationship to be proved for the space-time points belonging to the influence graph.
 By assumption, property~\eqref{eq:attractive-3} is true at time~$s = 0$.
 Assuming that~\eqref{eq:attractive-3} holds until time~$s-$ where~$s$ is the time of an update in the influence graph, we have the following cases. \vspace*{5pt} \\
{\bf Death} -- Assume first that there is a death mark~$\times$ at point~$(y, s)$.
 In this case, $(y, s-)$ is not in the influence graph but, regardless of the state at this space-time point, vertex~$y$ is empty at time~$s$ for both processes therefore
 the property to be proved is true at time~$s$. \vspace*{5pt} \\
{\bf Recovery} -- If there is a recovery mark~$\bullet$ at point~$(y, s)$ then~\eqref{eq:attractive-1} implies
 $$ \xi_s^1 (y) = \ind \,\{\xi_{s-}^1 (y) \neq 0 \} = \ind \,\{\xi_{s-}^2 (y) \neq 0 \} = \xi_s^2 (y) \quad \hbox{therefore} \quad \xi^1_s (y) \leq \xi^2_s (y). $$
{\bf Birth} --  If there is a birth arrow~$(y, s) \longrightarrow (z, s)$ then the configuration can only change if vertex~$z$ is empty (for both processes) in which case we have
 $$ \xi_s^1 (z) \ = \ \xi_{s-}^1 (y) \ \leq \ \xi_{s-}^2 (y) \ = \ \xi_s^2 (z) $$
 while the state at the other vertices remains unchanged. \vspace*{5pt} \\
{\bf Infection} -- If there is an infection arrow~$(y, s) \dashrightarrow (z, s)$ then the configuration can only change if both vertices~$y$ and~$z$ are occupied in which case we have
 $$ \xi_s^1 (z) \ = \ \max \,(\xi_{s-}^1 (y), \xi_{s-}^1 (z)) \ \leq \ \max \,(\xi_{s-}^2 (y), \xi_{s-}^2 (z)) \ = \ \xi_s^2 (z) $$
 while the state at the other vertices remains unchanged. \vspace*{5pt} \\
 In all four cases, property~\eqref{eq:attractive-3} is true at time~$s$ and obviously remains true until the time of the next update occurring in the influence graph so the
 result follows by induction.
\begin{table}[t]
\begin{center}
\begin{tabular}{ccp{240pt}}
\hline \noalign{\vspace*{2pt}}
 rate            & symbol                                   & effect on the stacked contact process \\ \noalign{\vspace*{1pt}} \hline \noalign{\vspace*{6pt}}
 1               & $\times$ at $x$ for all $x \in \Z^d$     & death at~$x$ when~$x$ is occupied \\ \noalign{\vspace*{3pt}}
 $\delta$        & $\bullet$ at $x$ for all $x \in \Z^d$    & recovery at~$x$ when~$x$ is infected \\ \noalign{\vspace*{3pt}}
 $\lambda_1 / N$ & $x \longrightarrow y$ for all $x \sim y$ & birth at vertex~$y$ when~$x$ is occupied and~$y$ is empty \\ \noalign{\vspace*{3pt}}
 $\lambda_2 / N$ & $x \dasharrow y$ for all $x \sim y$      & infection at vertex~$y$ when~$x$ is infected and~$y$ is in state~1 \\ \noalign{\vspace*{4pt}} \hline
\end{tabular}
\end{center}
\caption{\upshape{Graphical representation of the stacked contact process.
 The rates in the left column correspond to the different parameters of the independent Poisson processes, attached to either each vertex (first two rows) or
 each oriented edge connected two neighbors (last two rows).
 The number $N$ denotes the neighborhood size.}}
\label{tab:harris}
\end{table}
\end{proof} \\ \\
 Lemma~\ref{lem:attractive} will be used in the last section to prove survival of the infected hosts under the assumption of long range interactions.
 In the next lemma, we prove that the conclusion of Lemma~\ref{lem:attractive} remains true if we increase the infection parameter of the second process.
 This result shows that, as explained in the introduction, the birth and recovery parameters being fixed, there is at most one phase transition at a certain critical infection rate.
\begin{figure}[t]
\centering
\scalebox{0.40}{\input{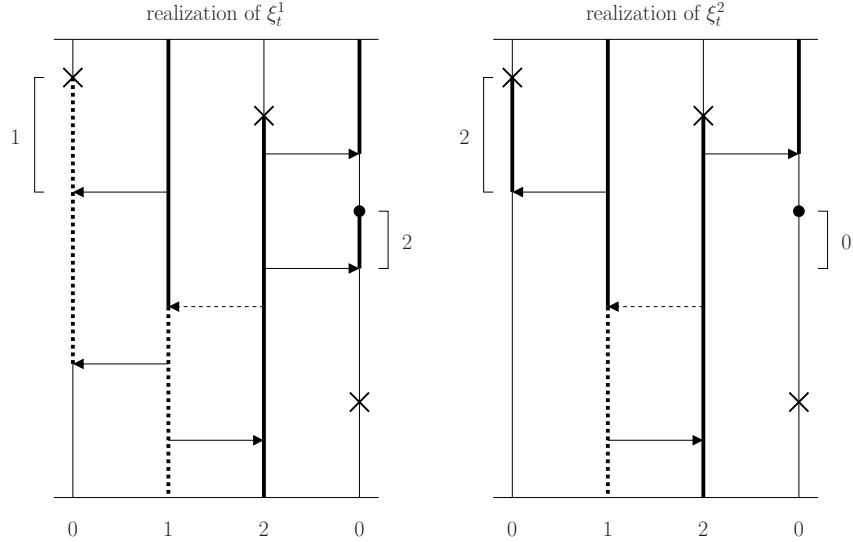}}
\caption{\upshape{Basic coupling of stacked contact processes with different birth rates.
 The solid thick lines represent space-time points occupied by an infected host while the dashed thick lines represent the space-time
 points occupied by a healthy host.
 The picture shows an example of realization for which the set of space-time point which are infected is not larger for the process with the larger birth rate.}}
\label{fig:coupling}
\end{figure}
\begin{lemma} --
\label{lem:monotone}
 Let~$\xi_t^1$ and~$\xi_t^2$ be two copies of the stacked contact process with the same birth and recovery rates but different infection rates: $\lambda_2^1 < \lambda_2^2$.
 Then, there is a coupling such that
 $$ \xi^1_t \ \preceq \ \xi^2_t \quad \hbox{whenever} \quad \xi^1_0 \ \preceq \ \xi^2_0. $$
\end{lemma}
\begin{proof}
 We construct the process~$\xi_t^2$ from the graphical representation obtained by adding infection arrows at the times of Poisson processes with
 parameter~$(\lambda_2^2 - \lambda_2^1) / \card N_x$ to the graphical representation used to construct the process~$\xi_t^1$.
 As previously, we only need to prove that, for any arbitrary space-time point~$(x, t)$, we have
\begin{equation}
\label{eq:monotone-1}
  (\xi_s^1 (y) \leq \xi_s^2 (y) \ \hbox{for all~$y \in \Z^d$ such that} \ (y, s) \in G^2 (x, t)) \quad \hbox{for all} \quad 0 \leq s \leq t
\end{equation}
 where~$G^2 (x, t)$ is the influence graph of the second process, which contains the influence graph of the first process.
 The same argument as in the proof of Lemma~\ref{lem:attractive} again implies that the result can be proved by checking that each of the successive events occurring along
 the influence graph of the second process preserves the relationship to be proved.
 Assume that~\eqref{eq:monotone-1} holds until time~$s-$ where~$s$ is the time of an update in the influence graph of the second process.
 In case there is an infection arrow~$(y, s) \dashrightarrow (z, s)$ in the graphical representation of the second process but not the first one, and assuming to avoid
 trivialities that~$y$ and~$z$ are occupied at time~$s-$,
 $$ \xi_s^1 (z) \ = \ \xi_{s-}^1 (z) \ \leq \ \xi_{s-}^2 (z) \ \leq \ \max \,(\xi_{s-}^2 (y), \xi_{s-}^2 (z)) \ = \ \xi_s^2 (z) $$
 while the state at the other vertices remains unchanged.
 In all the other cases, the same mark or arrow appears simultaneously in both graphical representation therefore the desired ordering is again preserved according
 to the proof of Lemma~\ref{lem:attractive}.
\end{proof} \\ \\
 To conclude this section, we note that the analysis of the mean-field model as well as numerical simulations of the stochastic spatial model also suggest some
 monotonicity of the survival probability of the infection with respect to the birth rate, i.e., for any fixed infection rate, recovery rate, and translation
 invariant initial distribution, the limiting probabilities
 $$ \begin{array}{l} \lim_{t \to \infty} \,P \,(\xi_t (x) = 2) \quad \hbox{for} \quad x \in \Z^d \end{array} $$
 are nondecreasing with respect to the birth rate~$\lambda_1$.
 This, however, cannot be proved using a basic coupling, i.e., constructing the process with the larger birth rate from the graphical representation obtained by
 adding birth arrows to the graphical representation used to construct the process with the smaller birth rate.
 Figure~\ref{fig:coupling} gives indeed an example of realization of this coupling for which the conclusion of the previous two lemmas does not hold.


\section{Proof of Theorem~\ref{th:coupling}}
\label{sec:coupling}

\indent The key idea to prove this theorem is to compare stacked contact processes with different birth and infection parameters.
 These processes are coupled through their graphical representation following the same approach as in Lemmas~\ref{lem:attractive} and~\ref{lem:monotone} though the coupling
 considered in this section is a little bit more meticulous.
 For a picture of the processes used for comparison in the next two lemmas, which respectively show the first and second parts of the theorem, we refer to Figure~\ref{fig:minmax}.
\begin{figure}[t]
\centering
\scalebox{0.40}{\input{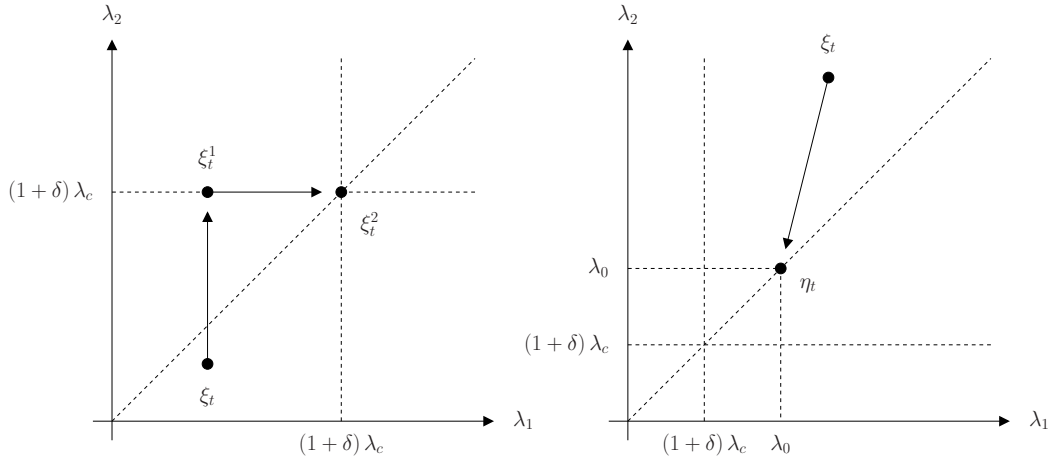}}
\caption{\upshape{Pictures related to the proof of Lemmas~\ref{lem:max}~and~\ref{lem:min}.}}
\label{fig:minmax}
\end{figure}
\begin{table}[t]
\begin{center}
\begin{tabular}{ccp{155pt}p{155pt}}
\hline \noalign{\vspace*{2pt}}
 rate                          & symbol                             & effect on the process~$\xi_t^1$
                                                                    & effect on the process~$\xi_t^2$ \\  \noalign{\vspace*{1pt}} \hline \noalign{\vspace*{6pt}}
 1                             & $\times$ at $y$ \ \                & death at~$y$ when~$y$ is occupied
                                                                    & death at~$y$ when~$y$ is occupied \\ \noalign{\vspace*{3pt}}
 $\delta$                      & $\bullet$ at $y$ \ \               & recovery at~$y$ when~$y$ is infected
                                                                    & recovery at~$y$ when~$y$ is infected \\ \noalign{\vspace*{3pt}}
 $\lambda_1 / N$               & $y \overset{1}{\longrightarrow} z$ & birth when~$y$ is occupied and~$z$ is empty and infection when~$y$ is infected and~$z$ is in state~1
                                                                    & birth when~$y$ is occupied and~$z$ is empty and infection when~$y$ is infected and~$z$ is in state~1 \\ \noalign{\vspace*{3pt}}
 $(\lambda_2 - \lambda_1) / N$ & $y \overset{2}{\longrightarrow} z$ & infection when~$y$ is infected and~$z$ is in state~1
                                                                    & birth when~$y$ is occupied and~$z$ is empty and infection when~$y$ is infected and~$z$ is in state~1 \\ \noalign{\vspace*{4pt}} \hline
\end{tabular}
\end{center}
\caption{\upshape{Coupling of the processes in the proof of Lemma~\ref{lem:max}.
 In the left column, $N$ is the neighborhood size.}}
\label{tab:max}
\end{table}
\begin{lemma} --
\label{lem:max}
 Assume that~$\max \,(\lambda_1, \lambda_2) \leq (1 + \delta) \,\lambda_c$. Then,
 $$ \begin{array}{l}
    \lim_{t \to \infty} \,P \,(\xi_t (x) = 2) = 0 \quad \hbox{for all} \quad x \in \Z^d. \end{array} $$
\end{lemma}
\begin{proof}
 In view of the monotonicity with respect to~$\lambda_2$ established in~Lemma~\ref{lem:monotone}, it suffices to prove that the stacked contact process~$\xi^1_t$ with
 birth and infection rates
 $$ \begin{array}{rcl}
        \hbox{birth rate} & = & \lambda_1 \ \leq \ \lambda_2 \ = \ (1 + \delta) \,\lambda_c \vspace*{2pt} \\
    \hbox{infection rate} & = & \lambda_2 \ = \ (1 + \delta) \,\lambda_c \end{array} $$
 dies out starting from any initial configuration.
 To prove this result, we couple this process with the stacked contact process~$\xi^2_t$ that has the same birth and infection rates
 $$ \hbox{birth rate} \ = \ \hbox{infection rate} \ = \ \lambda_2 \ = \ (1 + \delta) \,\lambda_c $$
 in which the set of infected hosts evolves according to a critical contact process.
 Both processes have in addition the same death rate one and recovery rate~$\delta$.
 Basic properties of Poisson processes imply that both processes can be constructed on the same probability space using the graphical representation described in Table~\ref{tab:max}.
 The next step is to show that, for this coupling and when both processes start from the same initial configuration, we have
\begin{equation}
 \label{eq:max-1}
   \xi^1_t (x) \ \leq \ \xi^2_t (x) \quad \hbox{for all} \quad (x, t) \in \Z^d \times \R_+.
\end{equation}
 Following the proofs of Lemmas~\ref{lem:attractive} and~\ref{lem:monotone}, it suffices to check that each of the successive events occurring along the influence graph
 of a space-time point~$(x, t)$ going forward in time preserves the relationship to be proved whenever it is true at time zero.
 Assume that~\eqref{eq:max-1} holds until time~$s-$ where~$s$ is the time of an update in the influence graph.
 In case a death mark or a recovery mark occurs at that time, the proof of Lemma~\ref{lem:attractive} implies that the property to be proved remains true after the update, while in
 case an arrow occurs, we have the following alternative. \vspace*{5pt} \\
{\bf Type~1 arrow} -- If there is an arrow~$(y, s) \overset{1}{\longrightarrow} (z, s)$ then
 $$ \xi_s^1 (z) \ = \ \max \,(\xi_{s-}^1 (y), \xi_{s-}^1 (z)) \ \leq \ \max \,(\xi_{s-}^2 (y), \xi_{s-}^2 (z)) \ = \ \xi_s^2 (z) $$
 while the state at the other vertices remains unchanged. \vspace*{5pt} \\
{\bf Type~2 arrow} -- If there is an arrow~$(y, s) \overset{2}{\longrightarrow} (z, s)$ then
 $$ \xi_s^1 (z) \ \leq \ \max \,(\xi_{s-}^1 (y), \xi_{s-}^1 (z)) \ \leq \ \max \,(\xi_{s-}^2 (y), \xi_{s-}^2 (z)) \ = \ \xi_s^2 (z) $$
 while the state at the other vertices remains unchanged. \vspace*{5pt} \\
 In all cases, property~\eqref{eq:max-1} is true at time~$s$ and remains true until the time of the next update occurring in the influence graph, which proves~\eqref{eq:max-1}.
 To conclude, we recall that the set of infected hosts in the second process evolves according to a critical contact process, which is known to die out starting from any initial
 configuration~\cite{bezuidenhout_grimmett_1990}.
 This and~\eqref{eq:max-1} imply that
 $$ \begin{array}{rcl}
    \lim_{t \to \infty} \,P \,(\xi_t (x) = 2) & \leq & \lim_{t \to \infty} \,P \,(\xi_t^1 (x) = 2) \vspace*{4pt} \\
                                              & \leq & \lim_{t \to \infty} \,P \,(\xi_t^2 (x) = 2) \ = \ 0 \quad \hbox{for all} \quad x \in \Z^d \end{array} $$
 which completes the proof.
\end{proof}
\begin{table}[t]
\begin{center}
\begin{tabular}{ccp{190pt}p{120pt}}
\hline \noalign{\vspace*{2pt}}
 rate                          & symbol                             & effect on the process~$\xi_t$
                                                                    & effect on the process~$\eta_t$ \\  \noalign{\vspace*{1pt}} \hline \noalign{\vspace*{6pt}}
 1                             & $\times$ at $y$ \ \                & death at~$y$ when~$y$ is occupied
                                                                    & death at~$y$ when~$y$ is occupied \\ \noalign{\vspace*{3pt}}
 $\delta$                      & $\bullet$ at $y$ \ \               & recovery at~$y$ when~$y$ is infected
                                                                    & death at~$y$ when~$y$ is occupied \\ \noalign{\vspace*{3pt}}
 $\lambda_0 / N$               & $y \overset{0}{\longrightarrow} z$ & birth when~$y$ is occupied and~$z$ is empty and \newline infection when~$y$ is infected and~$z$ in state~1
                                                                    & birth when~$y$ is occupied and~$z$ is empty \\ \noalign{\vspace*{3pt}}
 $(\lambda_1 - \lambda_0) / N$ & $y \overset{1}{\longrightarrow} z$ & birth when~$y$ is occupied and~$z$ is empty
                                                                    & none \\ \noalign{\vspace*{3pt}}
 $(\lambda_2 - \lambda_0) / N$ & $y \overset{2}{\longrightarrow} z$ & infection when~$y$ is infected and~$z$ in state~1
                                                                    & none \\ \noalign{\vspace*{4pt}} \hline
\end{tabular}
\end{center}
\caption{\upshape{Coupling of the processes in the proof of Lemma~\ref{lem:min}.
 In the left column, $N$ is the neighborhood size.}}
\label{tab:min}
\end{table}
\begin{lemma} --
\label{lem:min}
 Assume that~$\min \,(\lambda_1, \lambda_2) > (1 + \delta) \,\lambda_c$. Then,
 $$ \begin{array}{l}
    \liminf_{t \to \infty} \,P \,(\xi_t (x) = 2) > 0 \quad \hbox{for all} \quad x \in \Z^d. \end{array} $$
\end{lemma}
\begin{proof}
 First, we note that there exists~$\lambda_0$ such that
 $$ (1 + \delta) \,\lambda_c \ < \ \lambda_0 \ < \ \min \,(\lambda_1, \lambda_2). $$
 Then, we compare the stacked contact process~$\xi_t$ with the basic contact process~$\eta_t$ with birth parameter~$\lambda_0$ and death parameter~$1 + \delta$, and where it is assumed
 for convenience that occupied vertices are in state~2.
 Both processes can be constructed on the same probability space using the graphical representation described in Table~\ref{tab:min}, and we have
\begin{equation}
 \label{eq:min-1}
   \eta_t (x) \ \leq \ \xi_t (x) \quad \hbox{for all} \quad (x, t) \in \Z^d \times \R_+
\end{equation}
 provided this is satisfies at time zero.
 To prove~\eqref{eq:min-1}, it again suffices to check that each of the successive events occurring along the influence graph going forward in time preserves the relationship
 to be proved.
 The fact that death marks, recovery marks and type~0 arrows preserve this relationship follows from the same argument as in Lemma~\ref{lem:max}.
 For type~1 and type~2 arrows, we simply observe that the state at the tip of those arrows cannot decrease for the stacked contact process whereas they have no effect on the
 basic contact process.
 This shows that~\eqref{eq:min-1} holds.
 Finally, since~$\eta_t$ is a supercritical contact process, starting both processes with infinitely many vertices in state~2 and such that the ordering~\eqref{eq:min-1} is satisfied
 at time zero, we obtain
 $$ \begin{array}{l}
    \liminf_{t \to \infty} \,P \,(\xi_t (x) = 2) \ \geq \ \liminf_{t \to \infty} \,P \,(\eta_t (x) = 2) \ > \ 0 \quad \hbox{for all} \quad x \in \Z^d \end{array} $$
 which completes the proof.
\end{proof}


\section{Proof of Theorem~\ref{th:extinction} (extinction of the infection)}
\label{sec:extinction}

\begin{figure}[t]
\centering
\scalebox{0.40}{\input{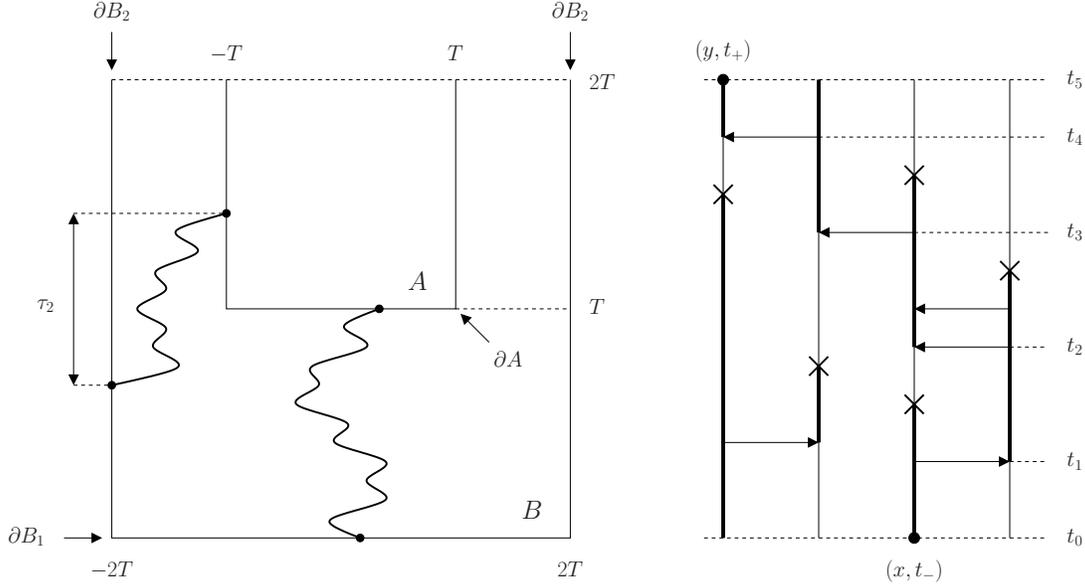}}
\caption{\upshape{Picture of the space-time regions~$A$ and~$B$ and example of an invasion path.}}
\label{fig:extinction}
\end{figure}

\indent Recall from Lemma~\ref{lem:monotone} that, the birth rate, death rate and recovery rate being fixed, the survival probability of the infection when
 starting from the configuration where all vertices are occupied by infected hosts is nondecreasing with respect to the infection parameter.
 This implies the existence of at most one phase transition between extinction of the infection and survival of the infection, and motivates the introduction of the critical value
 $$ \lambda_2^* \ = \ \lambda_2^* (\lambda_1, \delta) \ := \ \inf \,\{\lambda_2 \geq 0 : \hbox{the infection persists} \}. $$
 It directly follows from Lemma~\ref{lem:max} that
 $$ (1 + \delta) \,\lambda_c \ \leq \ \lambda_2^* (\lambda_1, \delta) \ \leq \ \infty \quad \hbox{when} \quad \lambda_1 \ \leq \ (1 + \delta) \,\lambda_c. $$
 To establish Theorem~\ref{th:extinction}, which states more generally that the critical infection rate is strictly positive for all possible values of the birth rate and the
 recovery rate, it suffices to show that, regardless of the initial configuration, the infection goes extinct whenever the infection rate is positive but sufficiently small.
 Referring to the left-hand side of Figure~\ref{fig:extinction}, the key to the proof is to show that, regardless of the state of the stacked contact process along the bottom
 and peripheral boundaries of the space-time box~$B$, the probability that the infection reaches the smaller box~$A$ is close to zero when boxes are large.
 From this, and covering the space-time universe with such boxes, we will deduce that the probability that a given infection path intersects $n$ boxes decreases exponentially
 with~$n$, which implies extinction of the infection.
 To make the argument precise, we first turn our picture into equations:
 let~$T$ be a large integer and consider the space-time boxes
 $$ A \ := \ [- T, T]^d \times [T, 2T] \quad \hbox{and} \quad B \ := \ [- 2T, 2T]^d \times [0, 2T]. $$
 We also define the bottom and peripheral boundaries of~$B$ as
 $$ \begin{array}{rcl}
    \partial B_1 & := & \{(x, t) \in B : t = 0 \} \vspace*{4pt} \\
    \partial B_2 & := & \{(x, t) \in B : \max_{j = 1, 2, \ldots, d} \,|x_j| = 2T \} \end{array} $$
 as well as the lower boundary of the smaller region~$A$ as
 $$ \begin{array}{l} \partial A \ := \ \{(x, t) \in A : t = T \ \hbox{or} \ \max_{j = 1, 2, \ldots, d} \,|x_j| = T \}. \end{array} $$
 In the next two lemmas, we collect upper bounds for the probability that the infection reaches the lower boundary of~$A$ starting from the bottom or peripheral boundary of~$B$.
 In the next lemma, we start with the process with infection rate zero, and extend the result to general infection rates in the subsequent lemma.
 To avoid cumbersome notations, we only prove these results in the presence of nearest neighbor interactions but our approach easily extends to any dispersal range.
\begin{lemma} --
\label{lem:invasion-1}
 For~$\lambda_2 = 0$ and regardless of the states in~$\partial B_1$ and~$\partial B_2$,
 $$ P \,(\xi_t (x) = 2 \ \hbox{for some space-time point} \ (x, t) \in A) \ \leq \ \exp (- a_1 T) $$
 for a suitable constant~$a_1 = a_1 (\lambda_1, \delta) > 0$.
\end{lemma}
\begin{proof}
 We write~$(x, t_-) \leadsto (y, t_+)$, and say that there is an {\bf invasion path} connecting both space-time points when there exist
 vertices and times
 $$ x_1 = x, x_2, \ldots, x_n = y \in \Z^d \qquad \hbox{and} \qquad t_- = t_0 < t_1 < \cdots < t_n = t_+ $$
 such that the following conditions hold:
\begin{itemize}
 \item For each~$j = 1, 2, \ldots, n$, we have
  $$ \begin{array}{l}
     \lim_{\,s \uparrow t_{j - 1}} \xi_s (x_j) = 0 \qquad \hbox{and} \qquad \xi_s (x_j) \neq 0 \ \ \hbox{for all} \ \ s \in [t_{j - 1}, t_j). \end{array} $$
 \item For each~$j = 1, 2, \ldots, n - 1$, there is a birth arrow~$x_j \to x_{j + 1}$ at time~$t_j$.
\end{itemize}
 We call the time increment~$t_+ - t_-$ the {\bf temporal length} of the invasion path.
 Note that if a space-time point is occupied then there must be an invasion path starting from time zero and leading to this point.
 In addition, this invasion path is unique.
 To prove the lemma, the first ingredient is to find upper bounds for the number of invasion paths that start at the bottom or peripheral boundary
 and intersect the smaller space-time region~$A$, as well as lower bounds for the temporal length of these invasion paths.
 The number of invasion paths, say~$X_1$, starting from~$\partial B_1$ is bounded by the number of vertices on this boundary, i.e.,
\begin{equation}
\label{eq:invasion-1}
  X_1 \ \leq \ (4T + 1)^d \quad \hbox{with probability one}
\end{equation}
 and the temporal length~$\tau_1$ of any of these paths satisfies
\begin{equation}
\label{eq:invasion-2}
  \tau_1 \ \geq \ T \quad \hbox{with probability one}.
\end{equation}
 Now, since the number~$X_2$ of invasion paths starting from~$\partial B_2$ is bounded by the number of birth arrows
 starting from this boundary and since birth arrows occur at rate~$\lambda_1$,
 $$ X_2 \ \preceq \ \bar X_2 \ := \ \poisson (2T \,(4T + 1)^{d - 1} \,\lambda_1). $$
 where~$\preceq$ means stochastically smaller than.
 In particular, standard large deviation estimates for the Poisson random variable give the following bound:
\begin{equation}
\label{eq:invasion-3}
  \begin{array}{l}
    P \,(X_2 > 4T \,(4T + 1)^{d - 1} \,\lambda_1) \vspace*{4pt} \\ \hspace*{40pt} \leq \
    P \,(\bar X_2 > 4T \,(4T + 1)^{d - 1} \,\lambda_1) \ = \ P \,(\bar X_2 > 2 \,E \,(\bar X_2)) \ \leq \ \exp (- a_2 T) \end{array}
\end{equation}
 for a suitable constant~$a_2 = a_2 (\lambda_1) > 0$.
 In addition, since invasion paths starting from the peripheral boundary must have at least~$T$ birth arrows to reach~$\partial A$,
 the temporal length~$\tau_2$ of any of these paths satisfies
 $$ \tau_2 \ \succeq \ \bar \tau_2 \ := \ \gammadist (T, \lambda_1) $$
 where~$\succeq$ means stochastically larger than.
 In particular, using again large deviation estimates but this time for the Gamma distribution, we deduce that
\begin{equation}
\label{eq:invasion-4}
  \begin{array}{rcl}
    P \,(\tau_2 < T / 2 \lambda_1) & \leq & P \,(\bar \tau_2 < T / 2 \lambda_1) \vspace*{4pt} \\
                                   &   =  & P \,(\bar \tau_2 < (1/2) \,E \,(\bar \tau_2)) \ \leq \ \exp (- a_3 T) \end{array}
\end{equation}
 for a suitable constant~$a_3 = a_3 (\lambda_1) > 0$.
 To deduce that, with probability close to one, none of the invasion paths can bring the infection into~$A$, we observe
 that recovery marks occur independently at each vertex at rate~$\delta$.
 This implies that the number of recovery marks along a given invasion path is a Poisson random variable with parameter~$\delta$ times
 the temporal length of this path.
 In particular, letting~$Z$ be the exponential random variable with rate~$\delta$ and using that the infection can reach the space-time
 region~$A$ only if there is at least one invasion path that does not cross any recovery mark, we deduce that
 $$ \begin{array}{l}
      P \,(\xi_t (x) = 2 \ \hbox{for some} \ (x, t) \in A) \vspace*{4pt} \\ \hspace*{20pt} \leq \
      P \,(\xi_t (x) = 2 \ \hbox{for some} \ (x, t) \in A \,| \,X_1 \leq (4T + 1)^d \ \hbox{and} \ X_2 \leq 4T \,(4T + 1)^{d - 1} \,\lambda_1) \vspace*{4pt} \\ \hspace*{50pt} + \
      P \,(X_1 > (4T + 1)^d) + P \,(X_2 > 4T \,(4T + 1)^{d - 1} \,\lambda_1) \vspace*{4pt} \\ \hspace*{20pt} \leq \
          (4T + 1)^d \ P \,(Z > T) + 4T \,(4T + 1)^{d - 1} \,\lambda_1 \ (P \,(\tau_2 < T / 2 \lambda_1) + P \,(Z > T / 2 \lambda_1)) \vspace*{4pt} \\ \hspace*{50pt} + \
      P \,(X_1 > (4T + 1)^d) + P \,(X_2 > 4T \,(4T + 1)^{d - 1} \,\lambda_1). \end{array} $$
 This and~\eqref{eq:invasion-1}--\eqref{eq:invasion-4} imply that, regardless of the states in~$\partial B_1$ and~$\partial B_2$,
 $$ \begin{array}{l}
      P \,(\xi_t (x) = 2 \ \hbox{for some} \ (x, t) \in A) \ \leq \
          (4T + 1)^d \ \exp (- \delta T) \vspace*{4pt} \\ \hspace*{20pt} + \
           4T \,(4T + 1)^{d - 1} \,\lambda_1 \ (\exp (- a_3 T) + \exp (- \delta T / 2 \lambda_1)) + \exp (- a_2 T) \ \leq \
         \exp (- a_1 T) \end{array} $$
 for a suitable constant~$a_1 = a_1 (\lambda_1, \delta)> 0$.
\end{proof}
\begin{lemma} --
\label{lem:invasion-2}
 For all~$\lambda_2 \geq 0$ and regardless of the states in~$\partial B_1$ and~$\partial B_2$,
 $$ P \,(\xi_t (x) = 2 \ \hbox{for some} \ (x, t) \in A) \ \leq \ \exp (- a_1 T) + 2T \,(4T + 1)^{d - 1} \,\lambda_2. $$
\end{lemma}
\begin{proof}
 Let~$X_3$ be the number of infection arrows that point at~$B$. Then,
\begin{equation}
\label{eq:invasion-5}
  X_3 \ = \ \poisson (2T \,(4T + 1)^{d - 1} \,\lambda_2)
\end{equation}
 and is independent of the position of the birth arrows, death marks and recovery marks.
 In addition, given that there is no infection arrow that points at~$B$, like in the previous lemma, the infection can reach the space-time
 region~$A$ only if there is at least one invasion path that does not cross any recovery mark.
 In particular, it follows from~Lemma~\ref{lem:invasion-1} and~\eqref{eq:invasion-5} that
 $$ \begin{array}{l}
      P \,(\xi_t (x) = 2 \ \hbox{for some} \ (x, t) \in A) \vspace*{4pt} \\ \hspace*{25pt} \leq \
      P \,(\xi_t (x) = 2 \ \hbox{for some} \ (x, t) \in A \,| \,X_3 = 0) + P \,(X_3 \neq 0) \vspace*{4pt} \\ \hspace*{25pt} \leq \
           \exp (- a_1 T) + 1 - \exp (- 2T \,(4T + 1)^{d - 1} \,\lambda_2) \vspace*{4pt} \\ \hspace*{25pt} \leq \
           \exp (- a_1 T) + 2T \,(4T + 1)^{d - 1} \,\lambda_2. \end{array} $$
 This completes the proof.
\end{proof} \\ \\
 Next, we compare the process properly rescaled with oriented site percolation:
 we cover the space-time universe with translations of the boxes~$A$ and~$B$ by letting
 $$ \begin{array}{rclcrcl}
      A (z, n) & := & (2 T z, n T) + A &  & \partial A (z, n) & := & (2 T z, n T) + \partial A \vspace*{4pt} \\
      B (z, n) & := & (2 T z, n T) + B &  & \partial B_i (z, n) & := & (2 T z, n T) + \partial B_i \quad \hbox{for} \quad i = 1, 2,  \end{array} $$
 for each site~$(z, n) \in \Z^d \times \Z_+$ and declare~$(z, n)$ to be
\begin{equation}
\label{eq:percolation-0}
  \begin{array}{rcl}
    \hbox{{\bf infected}} & \hbox{when} & \xi_t (x) = 2 \ \ \hbox{for some} \ (x, t) \in A (z, n) \vspace*{4pt} \\
    \hbox{{\bf healthy}}  & \hbox{when} & \xi_t (x) \neq 2 \ \ \hbox{for all} \ (x, t) \in A (z, n). \end{array}
\end{equation}
 The next lemma is the key ingredient to couple the set of infected sites with a subcritical oriented site percolation process where paths can move horizontally
 in all spatial directions and vertically going up following the direction of time.
\begin{lemma} --
\label{lem:percolation}
 For all~$\ep > 0$, there exist~$T < \infty$ and~$\lambda_2^* = \lambda_2^* (T) > 0$ such that
 $$ P \,((z_i, n_i) \ \hbox{is infected for} \ i = 1, 2, \ldots, m) \ \leq \ \ep^m \quad \hbox{for all} \quad \lambda_2 \leq \lambda_2^* $$
 whenever~$|z_i - z_j| \vee |n_i - n_j| \geq 3$ for all~$i \neq j$.
\end{lemma}
\begin{proof}
 Recalling~$a_1$ from the previous lemmas, we fix
\begin{equation}
\label{eq:percolation-1}
  \begin{array}{rcl}
                    T \ = \ T (a_1) & := & - (1/a_1) \ln (\ep / 2) \ < \ \infty \vspace*{4pt} \\
  \lambda_2^* \ = \ \lambda_2^* (T) & := & \ep \ (4T \,(4T + 1)^{d - 1})^{-1} \ > \ 0. \end{array}
\end{equation}
 Since the graphical representation of the process is translation invariant in both space and time, it follows from Lemma~\ref{lem:invasion-2} that,
 for the specific values given in~\eqref{eq:percolation-1},
 $$ \begin{array}{l}
      P \,((z, n) \ \hbox{is infected}) \ = \ P \,(\xi_t (x) = 2 \ \hbox{for some} \ (x, t) \in A (z, n)) \vspace*{4pt} \\ \hspace*{150pt}
        \leq \ \exp (- a_1 T) + 2T \,(4T + 1)^{d - 1} \,\lambda_2 \ \leq \ \ep \end{array} $$
 for all~$\lambda_2 \leq \lambda_2^*$ and all~$(z, n) \in \Z^d \times \Z_+$.
 Since in addition this bound holds regardless of the states in the bottom boundary and peripheral boundary of~$B (z, n)$ and that
 $$ B (z, n) \,\cap \,B (z', n') \ = \ \varnothing \quad \hbox{whenever} \quad |z - z'| \vee |n - n'| \geq 3 $$
 the lemma follows.
\end{proof} \\ \\
 To complete the proof of the theorem, we now turn the set of sites~$\Z^d \times \Z_+$ into a directed graph by adding the following collection of oriented edges:
 $$ \begin{array}{rcl}
      (z, n) \to (z', n') & \hbox{if and only if} & |z - z'| \vee |n - n'| \geq 3 \ \ \hbox{and} \ \ n \leq n' \vspace*{4pt} \\
                          & \hbox{if and only if} & B (z, n) \,\cap \,B (z', n') = \varnothing \ \ \hbox{and} \ \ n \leq n' \end{array} $$
 and define a percolation process with parameter~$\ep$ by assuming that
 $$ P \,((z_i, n_i) \ \hbox{is open for} \ i = 1, 2, \ldots, m) \ = \ \ep^m $$
 whenever~$|z_i - z_j| \vee |n_i - n_j| \geq 3$ for all~$i \neq j$.
 Then, there exists a critical value~$\ep_c > 0$ that only depends on the spatial dimension such that, for all parameters~$\ep < \ep_c$,
 the set of open sites does not percolate.
 See~\cite[section 4]{vandenberg_grimmett_schinazi_1998} for more details.
 In particular, calling {\bf wet} site a site that can be reached by a directed path of open sites starting at level~$n = 0$, we have
\begin{equation}
\label{eq:percolation-2}
 \begin{array}{l} \lim_{n \to \infty} \,P \,((z, n) \ \hbox{is wet}) \ = \ 0 \quad \hbox{for all} \quad z \in \Z^d \quad \hbox{when} \quad \ep < \ep_c. \end{array}
\end{equation}
 To deduce extinction of the infection, we fix~$\ep \in (0, \ep_c)$, and let~$T$ and~$\lambda_2^*$ be defined as in~\eqref{eq:percolation-1}
 for this specific value of~$\ep$.
 Then, it follows from Lemma~\ref{lem:percolation} that the interacting particle system and the percolation process can be coupled in such a way
 that the set of open sites dominates the set of infected sites provided the infection rate is less than~$\lambda_2^*$.
 Since in addition the infection cannot appear spontaneously, we have for this coupling
\begin{equation}
\label{eq:percolation-3}
  \{(z, n) : \xi_t (x) = 2 \ \hbox{for some} \ (x, t) \in A (z, n) \} \ \subset \ \{(z, n) : (z, n) \ \hbox{is wet} \}.
\end{equation}
 Combining~\eqref{eq:percolation-2}--\eqref{eq:percolation-3}, we conclude that for all~$x \in 2 T z + [- T, T]^d$
 $$ \begin{array}{l} \lim_{t \to \infty} \,P \,(\xi_t (x) = 2) \ \leq \ \lim_{n \to \infty} \,P \,((z, n) \ \hbox{is wet}) \ = \ 0. \end{array} $$
 This completes the proof of Theorem~\ref{th:extinction}.


\section{Proof of Theorem~\ref{th:survival} (survival of the infection)}
\label{sec:survival}

\begin{figure}[t]
\centering
\scalebox{0.40}{\input{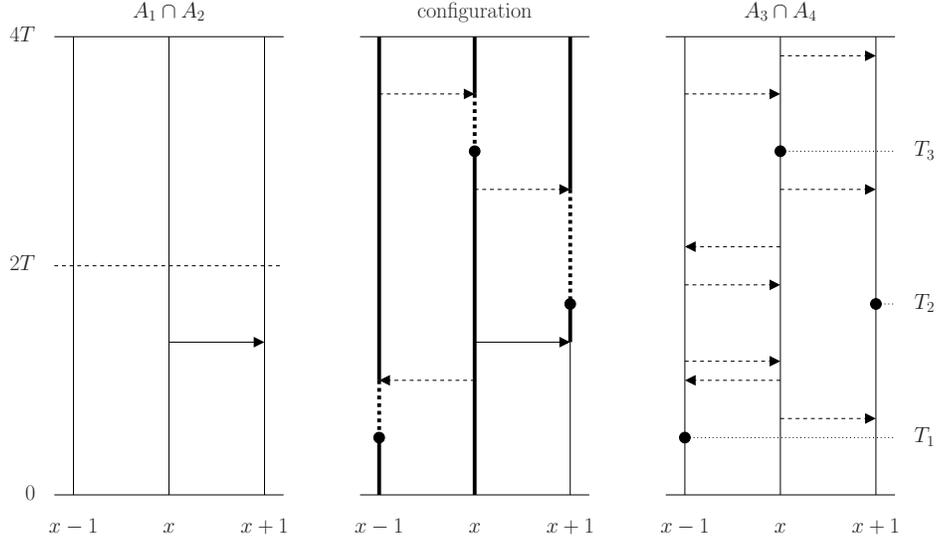}}
\caption{\upshape{Illustration of the event~$A_1 \cap A_2$ on the left and~$A_3 \cap A_4$ on the right.
 The picture at the center shows the configuration resulting from these events.
 Only the infection arrows that have an effect on the configuration are represented.
 The solid thick lines represent space-time points occupied by an infected host while the dashed thick lines represent the space-time
 points occupied by a healthy host.}}
\label{fig:survival}
\end{figure}
\indent In this section, we now study whether the critical infection rate~$\lambda_2^*$ is finite, meaning that the infection survives when the infection rate is sufficiently
 large, or infinite, meaning that the infection dies out for all infection rates, depending on the value of the birth and recovery rates.
 When the birth rate is subcritical, the host population dies out so the infection dies out as well, showing that the critical infection rate in this case is infinite.
 As pointed out in the introduction, we conjecture that the infection again dies out when the birth rate is barely supercritical due to the fact that the host population is
 too sparse to allow the infection to spread.
 Theorem~\ref{th:survival} states however that there exists a universal critical value~$\lambda_1^*$ such that
 $$ \lambda_2^* (\lambda_1, \delta) \ < \ \infty \quad \hbox{for all} \quad \lambda_1 \ > \ \lambda_1^* \quad \hbox{and} \quad \delta \ \geq \ 0. $$
 To prove this result, the first step is to take the birth rate large enough to ensure that the host population expands rapidly in order to provide some habitat for the infection.
 Then, we will show that, even when the recovery rate is large, the infection can invade this linearly growing set of hosts provided the infection rate is sufficiently large.
 These two steps are proved in the next two lemmas respectively.
 Note that the result is trivial when the critical value~$\lambda_1^*$ can be chosen depending on the recovery rate since Lemma~\ref{lem:min} directly implies that
 $$ \lambda_2^* (\lambda_1, \delta) \ \leq \ (1 + \delta) \,\lambda_c \ < \ \infty \quad \hbox{when} \quad \lambda_1 \ > \ (1 + \delta) \,\lambda_c $$
 In particular, an important component of the proof is the fact that~$\lambda_1^*$ will be ultimately a quantity that depends on the critical value of a certain oriented site
 percolation process but not on the recovery rate.
 Like in the previous section, we focus on the process with nearest neighbor interactions to avoid cumbersome notations but our approach easily extends to any dispersal range.
 We first establish the result in one dimension and will explain at the end of this section how to deduce the analog in higher dimensions.
 To state our next two lemmas, for all~$\ep > 0$, we let
\begin{equation}
\label{eq:parameter-invade}
 \begin{array}{rcl}
                      T \ = \ T (\ep) & := & - (1/12) \,\ln (1 - \ep / 4) \ > \ 0 \vspace*{4pt} \\
  \lambda_1^* \ = \ \lambda_1^* (\ep) & := & - (1/ T) \,\ln (\ep / 4) \ < \ \infty. \end{array}
\end{equation}
 For this specific time and this specific value of the birth rate, we have the following two lemmas that look respectively at the set of occupied and the
 set of infected vertices.
\begin{lemma} --
\label{lem:invade}
 For all~$\ep > 0$ and~$\lambda_1 > \lambda_1^*$,
 $$ \begin{array}{l}
      P \,(\xi_t (x - 1) \,\xi_t (x) \,\xi_t (x + 1) \neq 0 \ \hbox{for all} \ t \in (2T, 4T] \,| \,\xi_0 (x - 1) \,\xi_0 (x) \neq 0) > 1 - \ep / 2. \end{array} $$
\end{lemma}
\begin{proof}
 To begin with, we note that, given the conditioning, the event in the statement of the lemma occurs whenever the two events
 $$ \begin{array}{rcl}
     A_1 & := & \hbox{there are no death marks at vertices~$x - 1$ and~$x$ and~$x + 1$ by time~$4T$} \vspace*{2pt} \\
     A_2 & := & \hbox{there is a birth arrow $x \to x + 1$ by time~$2T$} \end{array} $$
 occur. Now, since death marks occur at each vertex at rate one,
\begin{equation}
\label{eq:invade-1}
  P \,(A_1) \ = \ P \,(\poisson (12T) = 0) \ = \ \exp (- 12T)
\end{equation}
 while, since birth arrows~$x \to x + 1$ occur at rate~$\lambda_1 / 2$,
\begin{equation}
\label{eq:invade-2}
  P \,(A_2) \ = \ P \,(\poisson (\lambda_1 T) \neq 0) \ = \ 1 - \exp (- \lambda_1 T).
\end{equation}
 Recalling~\eqref{eq:parameter-invade}, combining~\eqref{eq:invade-1}--\eqref{eq:invade-2} and using independence, we obtain
\begin{equation}
\label{eq:invade-3}
  \begin{array}{l}
    P \,(\xi_t (x - 1) \,\xi_t (x) \,\xi_t (x + 1) \neq 0 \ \hbox{for all} \ t \in (2T, 4T] \,| \,\xi_0 (x - 1) \,\xi_0 (x) \neq 0) \vspace*{4pt} \\ \hspace*{25pt} \geq \
    P \,(A_1 \cap A_2) \ = \ P \,(A_1) \,P \,(A_2) \ = \ \exp (- 12T) \,(1 - \exp (- \lambda_1 T)) \vspace*{4pt} \\ \hspace*{25pt} \geq \
      \exp (- 12T) \,(1 - \exp (- \lambda_1^* T)) \ = \ (1 - \ep / 4)^2 \ > \ 1 - \ep / 2 \end{array}
\end{equation}
 for all~$\lambda_1 > \lambda_1^*$.
 This completes the proof.
\end{proof}
\begin{lemma} --
\label{lem:infect}
 For all~$\ep > 0$ and~$\delta \geq 0$, there exists~$\lambda_2^* < \infty$ such that
 $$ P \,(\xi_{4T} (x) = \xi_{4T} (x + 1) = 2 \,| \,\xi_0 (x - 1) = \xi_0 (x) = 2) > 1 - \ep $$
 for all~$\lambda_1 > \lambda_1^*$ and~$\lambda_2 > \lambda_2^*$.
\end{lemma}
\begin{proof}
 First of all, we let~$N$ denote the random number of recovery marks that occur at any of the three vertices~$x - 1$ or~$x$ or~$x + 1$ by time~$4T$.
 Also, we let
 $$ 0 < T_1 < T_2 < T_3 < \cdots < T_N < 4T $$
 be the times at which these recovery marks appear.
 Given the conditioning in the statement of the lemma, vertices~$x$ and~$x + 1$ are infected at time~$4T$ whenever~$A_1 \cap A_2$ and
 $$ \begin{array}{rcl}
     A_3 & := & \hbox{between times~$T_j$ and~$T_{j + 1}$ for~$j = 1, 2, \ldots, N - 1$, there are} \vspace*{2pt} \\ &&
                \hbox{three infection arrows~$x - 1 \to x$ and~$x \to x - 1$ and~$x \to x + 1$} \vspace*{4pt} \\
     A_4 & := & \hbox{between times~$\max \,(2T, T_N)$ and~$4T$, there is an infection} \vspace*{2pt} \\ &&
                \hbox{arrow~$x - 1 \to x$ followed by an infection arrow~$x \to x + 1$} \end{array} $$
 all occur. To compute the probability of these events, let
 $$ X_i \ := \ \exponential (\lambda_2 / 2) \ \ \hbox{for} \ \ i = 1, 2, 3 \quad \hbox{and} \quad Z \ := \ \exponential (\delta) $$
 be independent.
 Now, we observe that, since recovery marks occur at each vertex at rate~$\delta$, there exists an integer~$n > 0$, fixed from now on, such that
 $$ P \,(N > n) \ = \ P \,(\poisson (3 \delta T) > n) \ < \ \ep / 8. $$
 In particular, there exists~$\lambda_2' < \infty$ such that, for all~$\lambda_2 > \lambda_2'$,
\begin{equation}
\label{eq:infect-1}
 \begin{array}{rcl}
   P \,(A_3) & \geq & P \,(A_3 \,| \,N \leq n) \,P \,(N \leq n) \vspace*{4pt} \\
             & \geq & (P \,(\max \,(X_1, X_2, X_3) < Z))^{n - 1} \,P \,(N \leq n) \vspace*{4pt} \\
             & \geq & (P \,(X_1 < Z))^{3n} \,P \,(N \leq n) \ = \ (\lambda_2 / (\lambda_2 + 2 \delta))^{3n} \,P \,(N \leq n) \vspace*{4pt} \\
             & \geq & (1 - \ep / 8)(1 - \ep / 8) \ > \ 1 - \ep / 4. \end{array}
\end{equation}
 Also, there exists~$\lambda_2'' < \infty$ such that, for all~$\lambda_2 > \lambda_2''$,
\begin{equation}
\label{eq:infect-2}
 \begin{array}{rcl}
   P \,(A_4) & = & \min \,(P \,(X_1 + X_2 < 2T), P \,(X_1 + X_2 < Z)) \vspace*{4pt} \\
             & = & \min \,((P \,(X_1 < 2T))^2, (P \,(X_1 < Z))^2) \vspace*{4pt} \\
             & = & \min \,((1 - \exp (- \lambda_2 T))^2, (\lambda_2 / (\lambda_2 + 2 \delta)^2) \ > \ 1 - \ep / 4. \end{array}
\end{equation}
 It follows from~\eqref{eq:invade-3}--\eqref{eq:infect-2} that
 $$ \begin{array}{l}
      P \,(\xi_{4T} (x) = \xi_{4T} (x + 1) = 2 \,| \,\xi_0 (x - 1) = \xi_0 (x) = 2) \vspace*{4pt} \\ \hspace*{40pt} \geq \
      P \,(A_1 \cap A_2 \cap A_3 \cap A_4) \ = \ P \,(A_1 \cap A_2) \,P \,(A_3 \cap A_4) \vspace*{4pt} \\ \hspace*{40pt} \geq \
      P \,(A_1 \cap A_2) \,(- 1 + P \,(A_3) + P \,(A_4)) \ > \ (1 - \ep / 2)^2 \ > \ 1 - \ep \end{array} $$
 for all~$\lambda_2 > \lambda_2^*$ where~$\lambda_2^* = \lambda_2^* (\ep, \delta) := \max \,(\lambda_2', \lambda_2'')$.
\end{proof} \\ \\
 To complete the proof of Theorem~\ref{th:survival}, we again turn~$\Z \times \Z_+$ into a directed graph, but now consider a different
 collection of oriented edges, namely
 $$ \begin{array}{rcl}
      (z, n) \to (z', n') & \hbox{if and only if} & |z - z'| = 1 \ \ \hbox{and} \ \ n' = n + 1. \end{array} $$
 We define a percolation process with parameter~$1 - \ep$ by assuming that
 $$ P \,((z_i, n_i) \ \hbox{is open for} \ i = 1, 2, \ldots, m) \ = \ (1 - \ep)^m $$
 whenever~$|z_i - z_j| \vee |n_i - n_j| \geq 2$ for all~$i \neq j$.
 For this process, called oriented site percolation, it is known that there exists a critical value~$\ep_c > 0$ such that, for
 all~$\ep < \ep_c$, the set of open sites percolates with probability one.
 See~\cite[section~10]{durrett_1984} for more details.
 To compare the stacked contact process with oriented site percolation, we declare site~$(z, n)$ to be
 $$ \begin{array}{rcl}
    \hbox{{\bf infected}} & \hbox{when} & \xi_{4nT} (2z) \ = \ \xi_{4nT} (2z + 1) \ = \ 2. \end{array} $$
 Now, we fix~$\ep \in (0, \ep_c)$, and let~$T (\ep)$ and~$\lambda_1^* (\ep)$ be defined as in~\eqref{eq:parameter-invade}.
 Calling again {\bf wet} site in the percolation process a site that can be reached by a directed path of open sites, it directly
 follows from Lemma~\ref{lem:infect} and~\cite[section~4]{durrett_1995} that the interacting particle system and the percolation
 process can be coupled in such a way that the set of infected sites dominates the set of wet sites provided the infection rate is
 larger than~$\lambda_2^*$.
 In particular, under the assumptions of Lemma~\ref{lem:infect},
 $$ \begin{array}{l} \liminf_{t \to \infty} \,P \,(\xi_t (x) = 2) \ \geq \ \liminf_{n \to \infty} \,P \,((z, n) \ \hbox{is wet}) \ > \ 0. \end{array} $$
 This completes the proof of Theorem~\ref{th:survival} for the one-dimensional process.
 To deal with the process in higher dimensions, we observe that adding birth and infection arrows can only increase the probability
 of the events~$A_2$, $A_3$ and~$A_4$ but does not affect the probability of the event~$A_1$ which only involves death marks.
 Since in addition, for all~$x \in \Z$, birth arrows
 $$ (x, 0, 0, \ldots, 0) \to (x \pm 1, 0, 0, \ldots, 0) $$
 for the~$d$-dimensional process with birth parameter~$d \lambda_1$ occur at rate~$\lambda_1 / 2$, and similarly for infection arrows,
 we deduce that, regardless of the spatial dimension,
 $$ \begin{array}{l}
      P \,(\xi_{4T} ((x, 0, \ldots, 0)) = \xi_{4T} ((x + 1, 0, \ldots, 0)) = 2 \,| \vspace*{4pt} \\ \hspace*{50pt}
           \xi_0 ((x - 1, 0, \ldots, 0)) = \xi_0 ((x, 0, \ldots, 0)) = 2) > 1 - \ep \end{array} $$
 for all~$\lambda_1 > d \lambda_1^*$ and~$\lambda_2 > d \lambda_2^*$.
 This shows that Lemma~\ref{lem:infect} holds in any dimensions provided the birth parameter and infection parameter are increased by the factor~$d$.
 The full theorem can be deduced as before using a coupling between the process and oriented percolation.


\section{Proof of Theorem~\ref{th:long-range} (long range interactions)}
\label{sec:long-range}

\indent This section is devoted to the proof of Theorem~\ref{th:long-range} which states that, for each set of parameters inside the coexistence region of the mean-field model,
 the stochastic process coexists as well provided the range of the interactions is sufficiently large.
 The key to the proof is a multiscale argument in order to couple the stacked contact process with supercritical oriented site percolation.
 To define this coupling, we first follow~\cite{durrett_zhang_2014} and introduce the box version of the process. \\


\noindent{\bf Box processes} --
 To control the environment so that the infection can spread, we need to introduce a process slightly smaller than the one considered in~\cite{durrett_zhang_2014}.
 For some fixed~$\ep_0 > 0$ which will be specified later, we let $l$ be the integer part of~$\ep_0 L$ and divide space into small boxes
\begin{equation}
\label{eq:boxes}
  \hat B_x \ := \ 2 l x + (-l, l]^d \quad \hbox{for all} \quad x \in \Z^d.
\end{equation}
 Note that this collection of boxes forms a partition of~$\Z^d$.
 To define the box version of a given process, the first step is to slightly reduce the interaction neighborhood of each vertex using the collection
 of boxes~\eqref{eq:boxes}.
 We define a new neighborhood of vertex~$x$ by setting
\begin{equation}
\label{eq:neighborhood}
 \begin{array}{l}
   \hat N_x \ := \ \{y \neq x : \|z_1 - z_2 \|_{\infty} \leq L \ \hbox{for all} \ z_1 \in \hat B_{x'} \ \hbox{and} \ z_2 \in \hat B_{y'} \} \end{array} 
\end{equation}
 where~$x'$ and~$y'$ are the unique vertices such that~$x \in \hat B_{x'}$ and~$y \in \hat B_{y'}$.
 In words, $\hat N_x$ is the largest set contained in~$N_x$ that can be written as a union of boxes.
\begin{lemma} --
\label{lem:neighborhood}
 For all~$x \in \Z^d$, we have
 $$ B_{\infty} (x, (1 - 4 \ep_0) L) \,\subset \,\hat N_x \,\subset \,N_x \quad \hbox{where} \quad B_{\infty} (x, r) := \{y \neq x : \|x - y \|_{\infty} \leq r \}. $$
\end{lemma}
\begin{proof}
 According to~\eqref{eq:neighborhood}, we have
 $$ \begin{array}{rcl}
      y \in \hat N_x & \quad \hbox{implies that} \quad & y \neq x \ \hbox{and} \ \|x - y \|_{\infty} \leq L \vspace*{4pt} \\
                     & \quad \hbox{implies that} \quad & y \in B_{\infty} (x, L) = N_x \end{array} $$
 which shows that~$\hat N_x \subset N_x$.
 Moreover, whenever
 $$ \|x - y \|_{\infty} \leq (1 - 4 \ep_0) L \quad \hbox{and} \quad z_1 \in \hat B_{x'} \quad \hbox{and} \quad z_2 \in \hat B_{y'} $$
 the triangle inequality implies that
 $$ \begin{array}{rcl}
    \|z_1 - z_2 \|_{\infty} & \leq & \|z_1 - x \|_{\infty} + \,\|x - y \|_{\infty} + \,\|y - z_2 \|_{\infty} \vspace*{4pt} \\
                            & \leq & 2l + (1 - 4 \ep_0) L + 2l \ = \ 4 \,\lfloor \ep_0 L \rfloor + (1 - 4 \ep_0) L \ \leq \ L. \end{array} $$
 In particular, recalling~\eqref{eq:neighborhood}, we have
 $$ \begin{array}{rcl}
      y \in B_{\infty} (x, (1 - 4 \ep_0) L) & \hbox{implies that} & 0 < \|x - y \|_{\infty} \leq (1 - 4 \ep_0) L \vspace*{4pt} \\
                                            & \hbox{implies that} & 0 < \|z_1 - z_2 \|_{\infty} \leq L \ \hbox{for all} \ (z_1, z_2) \in \hat B_{x'} \times \hat B_{y'} \vspace*{4pt} \\
                                            & \hbox{implies that} & y \in \hat N_x \end{array} $$
 therefore~$B_{\infty} (x, (1 - 4 \ep_0) L) \subset \hat N_x$ and the proof is complete.
\end{proof} \\ \\
 For every finite set~$A \subset \Z^d$ and~$\hat \xi : \Z^d \to \{0, 1, 2 \}$, we now let
 $$ \begin{array}{rrl}
      \hat f_j (A, \hat \xi) & := & (\card N_0)^{-1} \,\card \,\{y \in A : \hat \xi (y) = j \} \vspace*{4pt} \\
                             &  = & ((2L + 1)^d - 1)^{-1} \,\card \,\{y \in A : \hat \xi (y) = j \} \quad \hbox{for} \quad j = 0, 1, 2, \end{array} $$
 be the number of type~$j$ vertices in the set~$A$ rescaled by the size of the original interaction neighborhood.
 The box version~$\hat \xi_t$ of the stacked contact process is then defined as the process whose transition rates at vertex~$x$ are given by
 $$ \begin{array}{rclcrcl}
     0 \ \to \ 1 & \hbox{at rate} & \lambda_1 \,(\hat f_1 (N_x, \hat \xi) + \hat f_2 (N_x \setminus \hat N_x, \hat \xi)) & \quad & 1 \ \to \ 0 & \hbox{at rate} & 1 \vspace*{2pt} \\
     0 \ \to \ 2 & \hbox{at rate} & \lambda_1 \,\hat f_2 (\hat N_x, \hat \xi) & \quad & 2 \ \to \ 0  & \hbox{at rate} & 1 \vspace*{2pt} \\
     1 \ \to \ 2 & \hbox{at rate} & \lambda_2 \,\hat f_2 (\hat N_x, \hat \xi) & \quad & 2 \ \to \ 1  & \hbox{at rate} & \delta. \end{array} $$
 In words, hosts give birth and die, and infected hosts recover at the same rate as in the original process.
 However, an infected host at vertex~$x$ can only infect hosts or send infected offspring in the smaller neighborhood~$\hat N_x$.
 In particular, we have the following lemma.
\begin{lemma} --
\label{lem:box-process}
 There is a coupling of~$\xi_t$ and~$\hat \xi_t$ such that $\hat \xi_t \preceq \xi_t$ whenever $\hat \xi_0 \preceq \xi_0$.
\end{lemma}
\begin{proof}
 The stacked contact process~$\xi_t$ being constructed from the graphical representation introduced in the previous section, the box
 version~$\hat \xi_t$ can be constructed from the graphical representation obtained from the following two modifications:
\begin{enumerate}
 \item Remove all the infection arrows~$x \dasharrow y$ such that~$y \in N_x \setminus \hat N_x$. \vspace*{4pt}
 \item Label all the birth arrows~$x \longrightarrow y$ such that~$y \in N_x \setminus \hat N_x$ with a 1.
\end{enumerate}
 The box process is then constructed by assuming that
 $$ \hat \xi_{s-} (x) \neq 0, \ \ \hat \xi_{s-} (y) = 0 \ \ \hbox{and} \ \ (x, s) \overset{1}{\longrightarrow} (y, s) \quad \hbox{implies that} \quad \hat \xi_s (y) = 1 $$
 but otherwise using the same rules as for the original process.
 Since the birth arrows and the death marks occur at the same rate for both processes and have the same effect on whether vertices are
 empty or occupied, it follows that
 $$ \{x \in \Z^d : \hat \xi_t (x) \neq 0 \} \ = \ \{x \in \Z^d : \xi_t (x) \neq 0 \} \quad \hbox{for all} \quad t \geq 0. $$
 This can also be seen from adding the birth rates:
 $$ \begin{array}{l}
    \hat f_1 (N_x, \hat \xi) + \hat f_2 (N_x \setminus \hat N_x, \hat \xi) + \hat f_2 (\hat N_x, \hat \xi) \vspace*{4pt} \\ \hspace*{25pt} = \
    \hat f_1 (N_x, \hat \xi) + \hat f_2 (N_x, \hat \xi) \ = \ f_1 (x, \hat \xi) + f_2 (x, \hat \xi) \end{array} $$
 In particular, we only need to prove that
 $$ \{x \in \Z^d : \hat \xi_t (x) = 2 \} \ \subset \ \{x \in \Z^d : \xi_t (x) = 2 \} \quad \hbox{for all} \quad t \geq 0 $$
 which follows from the same argument as in the proof of Lemma~\ref{lem:monotone}.
\end{proof} \\


\begin{figure}[t]
\centering
\scalebox{0.40}{\input{block.pstex_t}}
\caption{\upshape{Picture of the block construction.}}
\label{fig:block}
\end{figure}

\noindent{\bf Block construction} --
 To complete the proof of the theorem, we compare the process with long range interactions with the same oriented site percolation process
 as in the previous section but using other space and time scales.
 Before going into the details of the proof, we start with a brief overview of the key steps which are illustrated in Figure~\ref{fig:block}.
 For simplicity, we only prove the result in one dimension but our approach easily extend to higher dimensions. Let
\begin{equation}
\label{UBD}
 \begin{array}{rcl}
    b & := & \lambda_2 \,(1 - 1 / \lambda_1) - \delta \vspace*{4pt} \\
  B_0 & := & 1 + \delta \vspace*{4pt} \\
    c & := & 8 \,(2B_0 + 1) / b \vspace*{4pt} \\
  T_1 & := & (1 + c) \,T + 2 \ep_0^{-1} \sqrt{T} \end{array}
\end{equation}
 where~$T$ is a large parameter to be fixed later and where~$\ep_0$ has been defined above to fix the size of the small boxes.
 Then, we declare site~$(z, n) \in \Z \times \Z_+$ to be {\bf infected} when
\begin{equation}
\label{eq:infected}
  \begin{array}{rcl}
    \card \{x \in \hat B_0 + z \sqrt{T} L : \xi_t (x) \neq 0 \} & \geq & l \,u_* \ = \ l \,(1 - 1 / \lambda_1) \vspace*{4pt} \\
    \card \{x \in \hat B_0 + z \sqrt{T} L : \xi_t (x) = 2 \}    & \geq & 2 l \,\exp (-T) \end{array}
\end{equation}
 for some~$t \in [2n T_1, (2n + 1) \,T_1]$.
 The proof combines three ingredients.
\begin{enumerate}
 \item {\bf Growth of the hosts} -- First, we show that, given the first event in~\eqref{eq:infected}, the population of hosts spreads
   so that, in the large space-time block drawn with a thick frame in the picture, the density of hosts in each small box is close to the mean-field
   equilibrium~$u_*$.
   This provides a habitat that the infection can invade.
   This is proved in Lemma~\ref{lem:contact}. \vspace*{4pt}
 \item {\bf Growth of the infection} -- The next step is to prove that, in this large space-time block, we can also
   increase the population of infected hosts as long as their density is low.
   This implies that at least one of the following two events must occur:
   \begin{enumerate}
    \item The density of infected hosts in the small box at the center of the large spatial block is larger than~$\exp (- T/2)$ at the fixed time~$(1 + c) \,T$.
    \item The density of infected hosts in some small box in the large spatial block is larger than~$T^{-1}$ at some random time before~$(1 + c) \,T$.
   \end{enumerate}
   This result is proved in Lemma~\ref{lem:growth}. \vspace*{4pt}
 \item {\bf Moving the infection} -- The last key ingredient is to show that infected hosts within a small box can quickly spread along a path of adjacent
   small boxes contained in the large spatial block without the density of the infected hosts decreasing too fast.
   This is used to prove that, each time event~(b) above occurs, we can re-center the infection to recover~(a) after a short time period, which is
   established rigorously in Lemmas~\ref{lem:move}--\ref{lem:target}.
\end{enumerate}
 Having step~1 allows us to repeat steps 2 and 3 a finite number of times so that for a long time there is always a small box in the large spatial block
 that has a reasonably large density of infected hosts.
 To complete the construction, we apply again step 3 to create two small boxes with the desired density of infected hosts at the center of the two adjacent blocks.
 To make this argument precise, we recall that hosts, either healthy or infected, evolve according to a basic contact process~$\eta_t$ with birth rate
 $\lambda_1 > 1$, death rate one and dispersal range $L$.
 We consider the interval
 $$ I_T \ := \ [- 2 \sqrt{T} L, 2 \sqrt{T} L] $$
 and let~$\bar \eta_t$ denote this contact process modified so that births outside~$2 I_T$ are not allowed, i.e., offspring sent outside this box are instantaneously killed.
 Finally, we let
 $$ \begin{array}{rcl}
                     \Delta_T & := & \{x \in \Z^d : \hat B_x \subset I_T \} \quad \hbox{and} \vspace*{4pt} \\
    \hat u_j (x, \bar \eta_t) & := & (\card \hat B_x)^{-1} \,\sum_{y \in \hat B_x} \ind \,\{\bar \eta_t (y) = j \}
                              \  = \ (2l)^{-1} \,\sum_{y \in \hat B_x} \ind \,\{\bar \eta_t (y) = j \} \end{array} $$
 be the fraction of type~$j$ vertices in box~$\hat B_x$ at time~$t$.
 In the next lemma, we prove that we can increase the population of hosts so that their density in each small box is close to the mean-field equilibrium~$u_*$.
\begin{lemma} --
\label{lem:contact}
 Assume that~$\hat u_1 (0, \eta_0) \geq (1/2) \,u_*$.
 Then, for all~$\rho > 0$,
 $$ \begin{array}{l}
    \lim_{L \to \infty} \,P \,(|\hat u_1 (x, \bar \eta_t) - u_*| < \rho \ \hbox{for all} \ t \in [T, 4(1+ c) \,T] \ \hbox{and} \ x \in \Delta_T) \ = \ 1 \end{array} $$
 whenever~$T$ is sufficiently large.
\end{lemma}
\begin{proof}
 First, we note that
\begin{equation}
\label{eq:contact-1}
  \begin{array}{l} E \,(\hat u_1 (x, \bar \eta_t)) \ = \ (2l)^{-1} \,\sum_{y \in \hat B_x} P \,(\bar \eta_t (y) = 1) \end{array} 
\end{equation}
 Now, since the contact process is self-dual, we also have
\begin{equation}
\label{eq:contact-2}
  P \,(\bar \eta_t (y) = 1) = P \,(\bar \eta^y_t \cap B \neq \varnothing) \quad \hbox{where} \quad B := \{z \in \hat B_0 : \bar \eta_0 (z) = 1 \}
\end{equation}
 and where~$\bar \eta^y_t$ is the contact process starting with a single individual at vertex~$y$.
 In addition, since the initial fraction of occupied vertices in box~$\hat B_0$ is larger than~$(1/2) \,u_*$,
 $$ \begin{array}{rcl}
    \card (B) & = & (2l)^d \,\hat u_1 (0, \bar \eta_0) \ \geq \ (2l)^d \,(1/2) \,u_* \vspace*{4pt} \\ 
              & \geq & (2L)^d \,\exp (- ((\lambda_1 - 1) / 2^d) \,t) \quad \hbox{for all $t$ large}. \end{array} $$
 In particular, it follows from~\cite[Lemma~3.4]{durrett_lanchier_2008} that, for all~$x \in \Delta_T$,
\begin{equation}
\label{eq:contact-3}
  \begin{array}{l} \lim_{L \to \infty} \,P \,(\bar \eta^y_t \cap B \neq \varnothing) \ = \ u_* \quad \hbox{for all~$t$ large and $y \in \hat B_x$}. \end{array}
\end{equation}
 Combining~\eqref{eq:contact-1}--\eqref{eq:contact-3}, we deduce that, for all~$t$ large and~$x \in \Delta_T$,
\begin{equation}
\label{eq:contact-4}
  \begin{array}{rcl}
   |E \,(\hat u_1 (x, \bar \eta_t)) - u_*| & = & |(2l)^{-1} \,\sum_{y \in \hat B_x} (P \,(\bar \eta_t (y) = 1) - u_*)| \vspace*{4pt} \\
                                           & = & |(2l)^{-1} \,\sum_{y \in \hat B_x} (P \,(\bar \eta^y_t \cap B \neq \varnothing) - u_*)| \ < \ \rho / 2 \end{array}
\end{equation}
 for all~$L$ sufficiently large.
 In view of~\eqref{eq:contact-4},
 $$ \begin{array}{l}
    \lim_{L \to \infty} \,P \,(|\hat u_1 (x, \bar \eta_t) - u_*| \geq \rho \ \hbox{for some} \ t \in [T, 4(1 + c) \,T] \ \hbox{and} \ x \in \Delta_T) \vspace*{4pt} \\ \hspace*{25pt} \leq \
    \lim_{L \to \infty} \,P \,(|\hat u_1 (x, \bar \eta_t) - E \,(\hat u_1 (x, \bar \eta_t))| \geq \rho / 2  \vspace*{4pt} \\ \hspace*{120pt} \hbox{for some} \ t \in [T, 4(1 + c) \,T] \ \hbox{and} \ x \in \Delta_T) \end{array} $$
 which is equal to zero according to~\cite[Lemma~3.5]{durrett_lanchier_2008}.
\end{proof} \\ \\
 To state our next lemma, we define
 $$ \begin{array}{rcl}
          \hat u_2 (x, \bar \xi_t) & := & (\card \hat B_x)^{-1} \,\sum_{y \in \hat B_x} \ind \,\{\bar \xi_t (y) = 2 \} \  = \ (2l)^{-1} \,\sum_{y \in \hat B_x} \ind \,\{\bar \xi_t (y) = 2 \} \vspace*{4pt} \\
    \hat u_{1 + 2} (x, \bar \xi_t) & := & (\card \hat B_x)^{-1} \,\sum_{y \in \hat B_x} \ind \,\{\bar \xi_t (y) \neq 0 \} \  = \ (2l)^{-1} \,\sum_{y \in \hat B_x} \ind \,\{\bar \xi_t (y) \neq 0 \} \end{array} $$
 where the process~$\bar \xi_t$ is the stacked contact process modified so that hosts outside~$I_T$ cannot get infected and offspring sent outside this interval instantaneously recover.
 According to Lemma~\ref{lem:attractive}, this process is dominated by the original stacked contact process.
 In the next lemma, we show that we can increase the population of infected hosts as long as their density is low.
\begin{lemma} --
\label{lem:growth}
 Assume that
 $$ \hat u_{1 + 2} (0, \bar \xi_0) \ > \ (1/2) \,u_* \quad \hbox{and} \quad \hat u_2 (0, \bar \xi_0) \ > \ \exp (-T). $$
 Then, for~$T$ large and in the limit as~$L \to \infty$, one of the following two events holds
\begin{enumerate}
 \item $\hat u_2 (0, \bar \xi_{(1 + c) T}) \geq \exp (-T/2)$, \vspace*{4pt}
 \item $\hat u_2 (x, \bar \xi_t) \geq T^{-1}$ for some~$t \in [T, (1 + c) \,T]$ and~$\hat B_x \subset I_T$.
\end{enumerate}
\end{lemma}
\begin{proof}
 To begin with, we define the two events
 $$ \begin{array}{rcl}
      A_0 & := & \{\hat u_2 (0, \bar \xi_T) > (1/2) \,\exp (-T \,(1 + B_0)) \} \vspace*{4pt} \\
      A_1 & := & \{|\hat u_{1 + 2} (x, \bar \xi_t) - u_*| < \rho \ \hbox{for all} \ t \in [T, 4(1+ c) \,T] \ \hbox{and} \ x \in \Delta_T \}. \end{array} $$
 According to~\cite[Lemma~3.1]{durrett_lanchier_2008}, the event~$A_0$ occurs with probability arbitrarily close to one when the range of the interactions is large, which,
 with Lemma~\ref{lem:contact}, gives
\begin{equation}
\label{eq:growth-1}
 \begin{array}{l} \lim_{L \to \infty} \,P \,(A_0) \ = \ \lim_{L \to \infty} \,P \,(A_1) \ = \ 1 \quad \hbox{when~$T$ is large}. \end{array}
\end{equation}
 Now, start a copy of the box version of the modified stacked contact process with~$\hat \xi_T = \bar \xi_T$ at time~$T$, both processes being constructed starting from this
 time from the same graphical representation as in Lemma~\ref{lem:box-process}. Let
 $$ \begin{array}{rcl} A_2 & := & \{\hat u_2 (x, \hat \xi_t) < T^{-1} \ \hbox{for all} \ t \in [T, (1+ c) \,T] \ \hbox{and} \ x \in \Delta_T \}. \end{array} $$
 Then, on the event~$A_1 \cap A_2$ and for all~$(x, t) \in [- \sqrt{T} L, \sqrt{T} L] \times [T, (1 + c) \,T]$ such that~$\hat \xi_t (x) = 2$, this type~2 particle dies at rate~$1 + \delta$
 and reproduces at rate at least
 $$ L^{-1} \,l \,(\lambda_1 \,(1 - u_* - \rho) + \lambda_2 \,(u_* - \rho)) $$
 to every~$\hat B_y \subset N_x$, which is a lower bound for the rate at which an infected host gives birth plus the rate at which it infects a healthy host.
 Since~$b = \lambda_2 \,(1 - 1 / \lambda_1) - \delta > 0$,
 $$ \begin{array}{l}
     (1 - 4 \ep_0)(\lambda_1 \,(1 / \lambda_1 - \rho) + \lambda_2 \,(1 - 1 / \lambda_1 - \rho)) \vspace*{4pt} \\ \hspace*{40pt} = \
     (1 - 4 \ep_0)(1 + \lambda_2 \,(1 - 1 / \lambda_1) - (\lambda_1 + \lambda_2) \,\rho) \vspace*{4pt} \\ \hspace*{40pt} = \
     (1 - 4 \ep_0)(1 + \delta + b - (\lambda_1 + \lambda_2) \,\rho) \ > \ 1 + \delta + b / 2 \ =: \ \beta \end{array} $$
 for some~$\rho, \ep_0 > 0$ fixed from now on.
 Now, we consider the basic contact process~$\zeta_t$ with birth rate~$\beta$, death rate~$1 + \delta$ and dispersal range~$L - 4l$, again modified so that births
 outside~$I_T$ are not allowed, i.e., offspring sent outside this box are instantaneously killed.
 Then, according to Lemmas~\ref{lem:neighborhood}--\ref{lem:box-process}, the set of infected hosts in the modified stacked contact process~$\bar \xi_t$ dominates the set of
 infected hosts in its box version~$\hat \xi_t$, which in turn dominates stochastically the set of occupied vertices in~$\zeta_t$ provided the corresponding sets are initially
 the same.
 This, together with~\cite[Lemma~4.3]{durrett_lanchier_2008} applied to the contact process~$\zeta_t$ when the dispersal range goes to infinity, implies that
 $$ \begin{array}{l} \lim_{L \to \infty} \,P \,(\{\hat u_2 (0, \bar \xi_{(1 + c) T}) < \exp(- T/2) \} \cap A_2 \,| \,A_0 \cap A_1) \ = \ 0. \end{array} $$
 Finally, using~\eqref{eq:growth-1} and again Lemma~\ref{lem:box-process}, we deduce that
 $$ \begin{array}{l}
    \lim_{L \to \infty} \,P \,(\hat u_2 (0, \bar \xi_{(1 + c) T}) \geq \exp(- T/2) \ \hbox{or} \vspace*{4pt} \\ \hspace*{100pt}
                               \hat u_2 (x, \bar \xi_t) \geq T^{-1} \ \hbox{for some} \ t \in [T, (1 + c) \,T] \ \hbox{and} \ \hat B_x \subset I_T) \vspace*{4pt} \\ \hspace*{10pt} \geq \
    \lim_{L \to \infty} \,P \,(\{\hat u_2 (0, \bar \xi_{(1 + c) T}) \geq \exp(- T/2) \} \cup A_2^c \,| \,A_0 \cap A_1) \vspace*{4pt} \\ \hspace*{10pt} = \
     1 - \lim_{L \to \infty} \,P \,(\{\hat u_2 (0, \bar \xi_{(1 + c) T}) < \exp(- T/2) \} \cap A_2 \,| \,A_0 \cap A_1) \ = \ 1. \end{array} $$
 This completes the proof.
\end{proof} \\ \\
 In the next two lemmas, we prove that, when event~(b) occurs, infected hosts within a small box can quickly spread along a path of adjacent small boxes without
 the density of the infected hosts decreasing too fast, which will allow us to re-center the infection and recover~(a) after a short time period.
 To make this precise, we let
 $$ H_{z, t} \ := \ 2l \,\hat u_2 (z, \xi_t) \ = \ \card \,\{x \in \hat B_z : \xi_t (x) = 2 \} $$
 be the number of infected hosts in the small box~$\hat B_z$ at time $t$.
 In the proofs, the stacked contact process is constructed from the graphical representation in Table~\ref{tab:move} where $\lambda := \min \,(\lambda_1,\lambda_2)$.
\begin{table}[t]
\begin{center}
\begin{tabular}{ccp{190pt}}
\hline \noalign{\vspace*{2pt}}
 rate                        & symbol                             & effect on the process~$\xi_t$ \\ \noalign{\vspace*{1pt}} \hline \noalign{\vspace*{6pt}}
 1                           & $\times$ at $y$ \ \                & death at~$y$ when~$y$ is occupied \\ \noalign{\vspace*{3pt}}
 $\delta$                    & $\bullet$ at $y$ \ \               & recovery at~$y$ when~$y$ is infected \\ \noalign{\vspace*{3pt}}
 $\lambda / N$               & $y \overset{0}{\longrightarrow} z$ & birth when~$y$ is occupied and~$z$ is empty and \newline infection when~$y$ is infected and~$z$ in state~1 \\ \noalign{\vspace*{3pt}}
 $(\lambda_1 - \lambda) / N$ & $y \overset{1}{\longrightarrow} z$ & birth when~$y$ is occupied and~$z$ is empty \\ \noalign{\vspace*{3pt}}
 $(\lambda_2 - \lambda) / N$ & $y \overset{2}{\longrightarrow} z$ & infection when~$y$ is infected and~$z$ in state~1 \\ \noalign{\vspace*{4pt}} \hline
\end{tabular}
\end{center}
\caption{\upshape{Graphical Representation used in Lemma~\ref{lem:move}}}
\label{tab:move}
\end{table}
\begin{lemma} --
\label{lem:move}
 There exist~$a > 0$ and~$L_0 < \infty$ such that, for all~$L > L_0$,
 $$ P \,(H_{z, 1} \geq a \,H_{0, 0} \,| \,H_{0, 0} \geq 2l \exp (-T)) \ > \ 1 - \exp (- \sqrt L) \quad \hbox{for} \quad z = -1, 0, 1. $$
\end{lemma}
\begin{proof}
 We begin with the case~$z = 0$, which is easier. Let
 $$ \begin{array}{l}
      G_0 \ := \ \{x \in \hat B_0 : \xi_0 (x) = 2 \ \hbox{and} \ \hbox{there is no death} \vspace*{2pt} \\ \hspace*{100pt}
             \hbox{or recovery marks on the segment} \ \{x \} \times [0, 1] \}. \end{array} $$
 Since each~$x \in G_0$ is occupied by an infected host at time 1, expressing~$\card G_0$ as a binomial random variable and using standard large deviation estimates,
 we obtain
 $$ \begin{array}{l}
      P \,(H_{0, 1} \geq a \,H_{0, 0} \,| \,H_{0, 0} \geq 2l \exp (- T)) \vspace*{4pt} \\ \hspace*{25pt} \geq \
      P \,(\card G_0 \geq a \,H_{0, 0} \,| \,H_{0, 0} \geq 2l \exp (- T)) \vspace*{4pt} \\ \hspace*{25pt} = \
      P \,(\binomial (H_{0, 0}, \,\exp (- (1 + \delta))) \geq a \,H_{0, 0} \,| \,H_{0, 0} \geq 2l \exp (- T)) \vspace*{4pt} \\ \hspace*{25pt} \geq \
      P \,(\binomial (H_{0, 0}, \,\exp (- (1 + \delta))) \geq (1/2) \exp (- (1 + \delta)) \,H_{0, 0} \,| \,H_{0, 0} \geq 2l \exp (- T)) \vspace*{4pt} \\ \hspace*{25pt} \geq \
      1 - \exp (- (l / 4) \,\exp (- T) \exp (- (1 + \delta))) \ \geq \ 1 - \exp (- \sqrt L) \end{array} $$
 for all~$a \leq (1/2) \,\exp (- (1 + \delta))$ and~$L$ sufficiently large.
 To prove the result when~$z = \pm 1$, we again consider the set~$G_0$ defined above as well as
 $$ G_z \ := \ \{y \in \hat B_z : \hbox{there is no death or recovery marks on} \ \{y \} \times [0, 1] \}. $$
 Then, writing~$x \to y$ to indicate that
 $$ \hbox{there is a type-0 arrow~$(x, s) \overset{0}{\longrightarrow} (y, s)$ for some~$s \in (1/2, 1)$}, $$
 we have the inclusion
\begin{equation}
\label{eq:move-1}
  \begin{array}{rrl}
    G_z' & := & \{y \in \Z^d : y \in G_z \ \hbox{and there exists~$x \in G_0$ such that $x \to y$} \} \vspace*{4pt} \\
         & \subset & \{y \in \hat B_z : \xi_1 (y) = 2 \} \ = \ H_{z, 1}. \end{array}
\end{equation}
 In addition, $\card G_0$ and~$\card G_z$ are equal in distribution to
\begin{equation}
\label{eq:move-2}
  \begin{array}{rcl}
    \card G_0 & = & \binomial (H_{0, 0}, \,\exp (- (1 + \delta))) \vspace*{4pt} \\
    \card G_z & = & \binomial (2l, \,\exp (- (1 + \delta))) \end{array}
\end{equation}
 while we have the conditional distribution
\begin{equation}
\label{eq:move-3}
  \begin{array}{l}
    P \,(\card G_z' = K \,| \,\card G_0 = K_0 \ \hbox{and} \ \card G_z = K_z) \vspace*{4pt} \\ \hspace*{100pt} = \
    P \,(\binomial (K_z, \,1 - \exp (- \lambda \,K_0 / 4L)) = K). \end{array}
\end{equation}
 In particular, letting~$a = (1/4l) \,K_z \,(1 - \exp (- \lambda \,K_0 / 4L)) \exp (T)$ where
 $$ K_0 \ := \ l \times \exp (-T) \exp (- (1 + \delta)) \quad \hbox{and} \quad K_z \ := \ l \times \exp (- (1 + \delta)) $$
 and combining~\eqref{eq:move-1}--\eqref{eq:move-3}, we deduce that
 $$ \begin{array}{l}
      P \,(H_{z, 1} \geq a \,H_{0, 0} \,| \,H_{0, 0} \geq 2l \exp (-T)) \vspace*{4pt} \\ \hspace*{20pt} \geq \
      P \,(G_z' \geq a \,H_{0, 0} \,| \,H_{0, 0} \geq 2l \exp (-T)) \vspace*{4pt} \\ \hspace*{20pt} \geq \
      P \,(G_z' \geq a \,H_{0, 0} \,| \,H_{0, 0} \geq 2l \exp (-T) \ \hbox{and} \ \card G_0 \geq K_0 \ \hbox{and} \ \card G_z \geq K_z) \vspace*{4pt} \\ \hspace*{50pt} \times \
      P \,(\card G_0 \geq K_0 \ \hbox{and} \ \card G_z \geq K_z \,| \,H_{0, 0} \geq 2l \exp (-T)) \vspace*{4pt} \\ \hspace*{20pt} \geq \
      P \,(\binomial (K_z, \,1 - \exp (- \lambda \,K_0 / 4L) \geq (1/2) \,K_z \,(1 - \exp (- \lambda \,K_0 / 4L))) \vspace*{4pt} \\ \hspace*{50pt} \times \
      P \,(\binomial (H_{0, 0}, \,\exp (- (1 + \delta))) \geq K_0 \,| \,H_{0, 0} \geq 2l \exp (-T)) \vspace*{4pt} \\ \hspace*{50pt} \times \
      P \,(\binomial (2l, \,\exp (- (1 + \delta))) \geq K_z) \vspace*{4pt} \\ \hspace*{20pt} \geq \
      1 - \exp (- (K_z / 8)(1 - \exp (- \lambda \,K_0 / 4L))) - \exp (- K_0 / 4) - \exp (- K_z / 4) \vspace*{4pt} \\ \hspace*{20pt} \geq \
      1 - \exp (- \sqrt L) \end{array} $$
 for all~$L$ sufficiently large.
\end{proof} \\ \\
 Applying~$O(\sqrt{T})$ times Lemma~\ref{lem:move}, we deduce the following result.
\begin{lemma} --
\label{lem:target}
 Let~$T_0 = \ep_0^{-1} \sqrt{T}$ and~$x \in \Z$ with~$|x| \leq T_0$. Then,
 $$ P \,(H_{x, T_0} \geq H_{0, 0} \,\exp (- T/4) \,| \,H_{0, 0} \geq 2l \exp (- T / 2)) \ \geq \ 1 - \exp (- L^{1/4}) $$
 for all~$T$ large and all $L$ larger than some finite~$L_1 \geq L_0$.
\end{lemma}
\begin{proof}
 Let~$n$ be the integer part of~$T_0$.
 Then, there exist
 $$ x_0 = 0, x_1, x_2, \ldots, x_n = x \quad \hbox{such that} \quad |x_{i + 1} - x_i| \leq 1 \quad \hbox{for} \quad i = 0, 1, \ldots, n - 1. $$
 Let~$T$ be sufficiently large so that~$a^n \geq \exp (- T / 4)$.
 Then, by applying repeatedly Lemma~\ref{lem:move} along the corresponding path of adjacent small boxes, we obtain
 $$ \begin{array}{l}
      P \,(H_{x, T_0} \geq H_{0, 0} \,\exp (- T/4) \,| \,H_{0, 0} \geq 2l \exp (- T / 2)) \vspace*{4pt} \\ \hspace*{25pt} \geq \
      P \,(H_{x, T_0} \geq a^n \,H_{0, 0} \,| \,H_{0, 0} \geq 2l \exp (- T / 2)) \vspace*{4pt} \\ \hspace*{25pt} \geq \
     \prod_{i = 0, 1, \ldots, n - 1} \ P \,(H_{x_{i + 1}, i + 1} \geq a \,H_{x_i, i} \,| \,H_{x_i, i} \geq 2l \,a^i \exp (- T / 2)) \vspace*{4pt} \\ \hspace*{25pt} \geq \
     (1 - \exp (- \sqrt L))^n \ \geq \ 1 - n \,\exp (- \sqrt L) \ \geq \ 1 - \exp (- L^{1/4}) \end{array} $$
 for all~$L$ sufficiently large.
\end{proof}
\begin{lemma} --
\label{lem:block}
 There exists~$T$ large such that
 $$ \begin{array}{l}
     \lim_{L \to \infty} \,P \,((1, 1) \ \hbox{is infected} \,| \,(0, 0) \ \hbox{is infected}) \ = \ 1. \end{array} $$
\end{lemma}
\begin{proof}
 To simplify the notation, we introduce
 $$ \begin{array}{rcl}
      \hat N_{1 + 2} (z, t) & := & \card \{x \in \hat B_0 + z \sqrt{T} L : \xi_t (x) \neq 0 \} \vspace*{4pt} \\
      \hat N_2 (z, t)       & := & \card \{x \in \hat B_0 + z \sqrt{T} L : \xi_t (x) = 2 \} \end{array} $$
 and observe that~$(z, n)$ is infected when
  $$ \hat N_{1 + 2} (z, t) \geq l \,u_* \quad \hbox{and} \quad \hat N_2 (z, t) \geq 2 l \,\exp (-T) \quad \hbox{for some} \quad t \in [2n T_1, (2n + 1) \,T_1]. $$
 First, we apply Lemma~\ref{lem:contact} with~$\rho = (1/2) \,u_*$ to get
\begin{equation}
\label{eq:block-1}
  \begin{array}{l}
    \lim_{L \to \infty} \,P \,(\hat N_{1 + 2} (1, t) \geq l \,u_* \ \hbox{for all} \ t \in [2T_1, 3T_1] \,| \,(0, 0) \ \hbox{is infected}) \vspace*{4pt} \\ \hspace*{25pt} \geq \
    \lim_{L \to \infty} \,P \,(|\hat u_{1 + 2} (x, \xi_t) - u_*| < (1/2) \,u_* \vspace*{4pt} \\ \hspace*{50pt}
    \hbox{for all} \ t \in [T, 4(1+ c) \,T] \ \hbox{and} \ x \in \Delta_T \,| \,\hat u_{1 + 2} (0, \xi_0) \geq (1/2) \,u_*) \ = \ 1 \end{array}
\end{equation}
 for all~$T$ large.
 This proves that the first condition for~$(1, 1)$ to be infected holds with probability close to one when the parameters~$T$ and~$L$ are large.
 To deal with the second condition, we let
 $$ \begin{array}{rcl}
      B_1 & := & \{\hat u_2 (0, \xi_{(1 + c) T}) \geq \exp (-T/2) \} \vspace*{4pt} \\
      B_2 & := & \{\hat u_2 (x, \xi_t) \geq T^{-1} \ \hbox{for some} \ t \in [T, (1 + c) \,T] \ \hbox{and} \ \hat B_x \subset I_T \} \end{array} $$
 be the two events introduced in Lemma~\ref{lem:growth}.
 According to this lemma,
 \begin{equation}
\label{eq:block-2}
  \begin{array}{l}
    \lim_{L \to \infty} \,P \,(B_1 \cup B_2 \,| \,(0, 0) \ \hbox{is infected}) \ = \ 1 \end{array}
\end{equation}
 for all~$T$ large.
 In case~$B_1$ occurs, we do nothing, whereas in case~$B_2 \setminus B_1$ occurs, we use Lemma~\ref{lem:target} to re-center the infected hosts towards zero.
 More precisely, since on the event~$B_2$ there is a small box with a fraction of infected hosts exceeding~$T^{-1}$ which is distance at most~$2 \sqrt{T} L$ from
 the origin, this lemma implies that
 $$ \begin{array}{l}
    \lim_{L \to \infty} \,P \,(\hat N_2 (0, t + T_0) \geq 2 l \,\exp (-T/2) \,| \,B_2) \vspace*{4pt} \\ \hspace*{50pt} = \
    \lim_{L \to \infty} \,P \,(H_{0, t + T_0} \ \geq \ 2 l \,\exp (-T/2) \,| \,B_2) \vspace*{4pt} \\ \hspace*{50pt} \geq \
    \lim_{L \to \infty} \,P \,(H_{0, t + T_0} \ \geq \ 2 l \,T^{-1} \exp (-T / 4) \,| \,B_2) \ = \ 1 \end{array} $$
 for all~$T$ large.
 Recalling~\eqref{eq:block-2} and the definition of~$B_1$, it follows that
\begin{equation}
\label{eq:block-3}
 \begin{array}{l}
    \lim_{L \to \infty} \,P \,(\hat N_2 (0, s) \geq 2 l \,\exp (-T/2) \ \hbox{for some} \ s \in [T, (1 + c) \,T + T_0] \,| \,(0, 0) \ \hbox{is infected}) \vspace*{4pt} \\ \hspace*{20pt} \geq \
    \lim_{L \to \infty} \,P \,(\hat N_2 (0, s) \geq 2 l \,\exp (-T/2) \ \hbox{for some} \ s \in [T, (1 + c) \,T + T_0] \,| \,B_1 \cup B_2) \vspace*{4pt} \\ \hspace*{80pt}
      P \,(B_1 \cup B_2 \,| \,(0, 0) \ \hbox{is infected}) \ = \ 1 \end{array}
\end{equation}
 Combining~\eqref{eq:block-1} and~\eqref{eq:block-3}, we deduce
\begin{equation}
\label{eq:block-4}
 \begin{array}{l}
    \lim_{L \to \infty} \,P \,(\hat N_{1 + 2} (0, s) \geq l \,u_* \ \hbox{and} \ \hat N_2 (0, s) \geq 2 l \,\exp (-T/2) \vspace*{4pt} \\ \hspace*{70pt}
    \hbox{for some} \ s \in [T, (1 + c) \,T + T_0] \,| \,(0, 0) \ \hbox{is infected})  \ = \ 1 \end{array}
\end{equation}
 Now, define~$s_0 := 0$ and recursively for all~$i > 0$,
 $$ \begin{array}{l}
      s_i \ := \ \inf \,\{s : \hat N_{1 + 2} (0, s) \geq l \,u_* \ \hbox{and} \ \hat N_2 (0, s) \geq 2 l \,\exp (-T/2) \vspace*{4pt} \\ \hspace*{100pt}
    \hbox{for some} \ s \in [s_{i - 1} + T, s_{i - 1} + (1 + c) \,T + T_0] \} \end{array} $$
 where we use the convention~$\inf \varnothing = \infty$.
 Now, let~$n := \lfloor 2 (1 + c) + 1 \rfloor + 1$ and note that
 $$ 2T_1 / T \ = \ (2 (1 + c) \,T + 4 T_0) / T \ \leq \ 2 (1 + c) + 1 \ \leq \ n \quad \hbox{for all $T$ large}. $$
 Since~$n$ does not depend on~$T$ and~$L$, the strong Markov property and~\eqref{eq:block-4} imply that
\begin{equation}
\label{eq:block-5}
 \begin{array}{l}
   \lim_{L \to \infty} \,P \,(s_n < \infty \,| \,(0, 0) \ \hbox{is infected}) \vspace*{4pt} \\ \hspace*{40pt} = \
   \prod_{i = 0, 1, \ldots, n - 1} \ P \,(s_{i + 1} < \infty \,| \,s_i < \infty \ \hbox{and} \ (0, 0) \ \hbox{is infected}) \ = \ 1. \end{array}
\end{equation}
 In addition, on the event that~$s_n < \infty$, we have
 $$ s_{i + 1} - s_i \in [T, (1 + c) \,T + T_0] = [T, T_1 - T_0] \quad \hbox{for} \quad i = 0, 1, 2, \ldots, n - 1 $$
 therefore there exists~$i \leq n$ such that~$s_i \in [2 T_1, 3 T_1 - T_0]$.
 This together with \eqref{eq:block-5} implies that
\begin{equation}
\label{eq:block-6}
 \begin{array}{l}
    \lim_{L \to \infty} \,P \,(\hat N_{1 + 2} (0, s) \geq l \,u_* \ \hbox{and} \ \hat N_2 (0, s) \geq 2 l \,\exp (-T/2) \vspace*{4pt} \\ \hspace*{80pt}
    \hbox{for some} \ s \in [2T_1, 3 T_1 - T_0] \,| \,(0, 0) \ \hbox{is infected}) \vspace*{4pt} \\ \hspace*{40pt} \geq \
    \lim_{L \to \infty} \,P \,(s_n < \infty \,| \,(0, 0) \ \hbox{is infected}) \ = \ 1. \end{array}
\end{equation}
 Applying again Lemma~\ref{lem:target}, we obtain
\begin{equation}
\label{eq:block-7}
 \begin{array}{l}
    \lim_{L \to \infty} \,P \,(\hat N_{1 + 2} (1, s + T_0) \geq l \,u_* \ \hbox{and} \ \hat N_2 (1, s + T_0) \geq 2 l \,\exp (-T) \vspace*{4pt} \\ \hspace*{80pt}
    \,| \,\hat N_{1 + 2} (0, s) \geq l \,u_* \ \hbox{and} \ \hat N_2 (0, s) \geq 2 l \,\exp (-T/2) \vspace*{4pt} \\ \hspace*{80pt}
    \hbox{for some} \ s \in [2T_1, 3 T_1 - T_0]) \ = \ 1. \end{array}
\end{equation}
 The lemma follows by observing that the conditional probability in the statement is larger than the product of~\eqref{eq:block-6} and~\eqref{eq:block-7}. 
\end{proof} \\ \\
 Since the graphical representation is translation invariant in space and time, Lemma~\ref{lem:block} implies that the process can be coupled
 with supercritical oriented site percolation in which a way that the set of infected sites dominates the set of wet sites.
 Theorem~\ref{th:long-range} can then be deduced using the same argument as in the previous section. \\


\noindent\textbf{Acknowledgment}.
 The authors would like to thank Rick Durrett for suggesting the problem and for his comments on a preliminary version of this work.


\end{document}